\documentclass[12pt]{amsart}

\usepackage{amsmath}
\usepackage{amssymb, amsfonts,amscd,verbatim}
\usepackage{hyperref}
\usepackage{mathtools}
\usepackage{graphicx}
\usepackage{cleveref}
\usepackage{enumitem}
\usepackage[all]{xy}
\usepackage{tikz}
\usetikzlibrary{matrix}

\usepackage{booktabs}

\newtheorem{thm}{Theorem}[section]
\newtheorem{prop}[thm]{Proposition}
\newtheorem{cor}[thm]{Corollary}
\newtheorem{lemma}[thm]{Lemma}
\newtheorem{defn}[thm]{Definition}
\newtheorem{rem}[thm]{Remark}
\numberwithin{equation}{section}

\renewcommand{\k}{\ensuremath{\mathbf{k}}}
\newcommand{\xkn}{\ensuremath{X(\mathbf{k}^n)}}
\newcommand{\kkn}{\ensuremath{K(\mathbf{k}^n)}}
\def\Z{{\mathbb Z}}
\def\N{{\mathbb N}}
\def\F{{\mathbb F}}

\newcommand{\lra}{\longrightarrow}

\DeclareMathOperator{\spn}{span}

\DeclareMathOperator{\Tor}{Tor}
\DeclareMathOperator{\Lk}{Lk}
\DeclareMathOperator{\rank}{rank}
\DeclareMathOperator{\cat}{cat}
\DeclareMathOperator{\cp}{cup}

\begin{document}

\title{Universal complexes in toric topology}

\author{Djordje Barali\'{c}, Ale\v s Vavpeti\v c, Aleksandar Vu\v{c}i\'{c}}
\address{ \scriptsize{Mathematical Institute SASA, Belgrade, Serbia }}
\email{djbaralic@mi.sanu.ac.rs}
\address{Faculty of Mathematics and Physics, University of Ljubljana, Slovenia and Institute of Mathematics, Physics and Mechanics, Ljubljana, Slovenia}
\email{ales.vavpetic@fmf.uni-lj.si}
\address{\scriptsize{University of Belgrade, Faculty of Mathematics, Belgrade, Serbia}}
\email{aleksandar.vucic@matf.bg.ac.rs}

\subjclass[2020]{Primary  57S12, 55U05, Secondary 57Q70, 13F55.}

%\keywords{commutator, cycle, permutation, alternating group}

%\subjclass[2010]{20B05}

%\thanks{This research was supported by the Slovenian Research
%Agency grants P1-0292-0101 and J1-2057-0101.}

\begin{abstract}
We study combinatorial and topological properties of the universal complexes $X(\F_p^n)$ and $K(\F_p^n)$ whose simplices are certain unimodular subsets of $\F_p^n$. We calculate their $\mathbf f$-vectors and their Tor-algebras, show that they are shellable but not shifted, and find their applications in toric topology and number theory.
We showed that the Lusternick-Schnirelmann category of the moment angle complex of $X(\F_p^n)$ is $n$, provided $p$ is an odd prime, and the Lusternick-Schnirelmann category of the moment angle complex of $K(\F_p^n)$ is $[\frac n 2]$. Based on the universal complexes, we introduce the Buchstaber invariant $s_p$ for a prime number $p$.
\end{abstract}

\maketitle

\section{Introduction}

Toric topology is a contemporary active research area of mathematics in deep interaction with equivariant topology, commutative algebra, algebraic and symplectic geometry, combinatorics, and other disciplines. In the core of toric topology are topological spaces with torus actions defined in combinatorial terms for which simplicial complexes and polytopes play an important role.

Davis and Januszkiewicz, in their seminal paper \cite{DJ} motivated the study of quasitoric manifolds and small covers as topological generalizations of nonsingular projective toric variety and its real analog, respectively. Until the end of the 20th century, several other topological constructions appeared as a result of efforts to develop a topological generalization of toric varieties from algebraic geometry. Moment angle complexes and moment angle manifolds were major milestones in the development of toric topology that lead to penetration of categorical constructions in the field and further generalizations to the polyhedral product spaces. The interest in their topological properties, cohomology opened space for ideas from homotopy theory, commutative algebra, and combinatorics which provided new points of view. A substantial overview of the most significant results in this area may be found in the remarkable monograph \cite{BPbook} from 2015.

Universal complex is a simplicial complex whose maximal simplices are bases of the lattices of $\mathbb{Z}^n$ or $\F_2^n$. They are closely related to the fundamental geometric object known as Tits building studied in homological algebra, cohomology of groups, number theory, etc. For theory and applications of buildings, we refer the reader to monograph \cite{AbraBrown}. In their paper \cite{DJ}, Davis and Januszkiewicz first realized the importance of universal complexes in generalizing toric varieties. Indeed, for these purposes, they defined another universal complex of unimodular sequences in $\mathbb{Z}^n$ or $\F_2^n$ and described their close relation. They pointed out their remarkable homotopical properties and relation to Cohen-Macaulay complexes. In the following decades, toric topologists used these complexes in studying the Buchstaber invariant of simplicial complex or various classification problems in toric topology.

Toric topology found a lot of applications outside mathematics, like robotics, chemistry and the theory of fullerenes, and topological data analysis. New applications stipulate further investigation and calculation of more complex topological indexes like the Lusternik–Schnirelmann category and various types of topological complexity, see \cite{BebenGrbic3}.

The universal complexes and their topological properties were studied by Barali\'{c}, Grbi\'{c}, Vavpeti\v{c} and Vu\v{c}i\'{c} in \cite{BGVV}. They introduced mod $p$ Buchstaber invariant and found some application in number theory. The first version on arXiv in 2017 offered proof based on discrete Morse that the universal complex has the homotopy type of a wedge of spheres, but the argument was inaccurate. The authors provided direct evidence in the second version of \cite{BGVV} from 2020. However, the main result is just a corollary of \cite[Theorem~E]{Church} published in 2019.

The presented research in this article focuses on the toric topology of the universal complexes $X(\F_p^n)$ and $K(\F_p^n)$ and their Tor-algebras. Various aspects of the cohomology rings of the moment angle complexes of $X(\F_p^n)$ and $K(\F_p^n)$ have been studied. The results concerning $f$-vectors of the universal complexes and their fundamental properties from \cite{BGVV} are included in the paper. The section on the mod $p$ Buchstaber invariant and its properties from the unpublished arXiv version \cite{BGVV} is also extended and included in the article.

In Section 2, we review basic facts about the universal complexes with an emphasis on their universality property. We study universal complexes $X(\F_p^n)$ and $K(\F_p^n)$ and determine its f-vectors, where $p$ is prime. Here we show that the universal complexes are homotopy equivalent to a wedge of $(n-1)$-spheres.

Tor-algebras of the universal complexes are studied in Sections 3 and 4. The corresponding chain complexes have nice combinatorial descriptions, which are very useful when applying the algebraic Morse theory. We construct perfect Morse matching, which allows us to calculate the bigraded Betti numbers of the universal complexes and to prove that the Tor-algebras have no torsion.

Section 4 studies products in the Tor-algebras and the cup length in them. Besides the cup lengths, we determined the Lusternik–Schnirelmann category of the moment angle complexes over $X(\F_p^n)$ and $K(\F_p^n)$. It is a surprising fact that the two complexes sharing a lot of common features have distinct categories.

In Section 5, based on the universal complexes, we introduce the Buchstaber invariant $s_p$ for a prime number $p$. We extend classical results obtained by Ayzenberg, Erokhovets, Fukukawa, and Masuda on real and standard Buchstaber invariant to $s_p$ but also establish some new results that open door for further investigation.

\section{Universal complexes and their properties}

\begin{defn}
Let $\k$ be the field $\mathbb{F}_p$
or the ring $\mathbb{Z}$. A set $\{ v_1, \ldots , v_m \}$ of elements in $\k^n$ is called \emph{unimodular} if $\spn\{v_1,\ldots ,v_m\}$ is a direct summand of $\k^n$ of dimension $m$. The simplicial complex $\xkn$ on unimodular vertices $v_i\in \k^n$ consists of all unimodular subsets of $\k^n$.  Note that a subset of a unimodular set is itself unimodular.
\end{defn}

Another simplicial complex closely related to $ \xkn$ is the simplicial complex $\kkn$.

Define \textit{a line through the origin}, or shortly \textit{a line}, to be a 1-dimensional unimodular subspace of $\k^n$, that is, a submodule of rank $1$ which is a direct summand of $\k^n$. More precisely, for $\k = \F_p$, a line is an $1$-dimensional subspace of $\F_p^n$, while in the case $\k=\Z$, a line is an infinite cyclic subgroup of $\Z^n$ generated by an element which is an element of some basis of $\Z^n$. Denote by $P\k^n$ the set of all such lines in $\k^n$.
\begin{defn}
The simplicial complex $\kkn$ on the vertex set  $P\k^n$ is defined so that its $m$-simplex is a set of lines $\{l_1, \dots, l_{m+1}\}$, $l_i\in P\k^n$, which span an $(m+1)$-dimensional unimodular subspace of $\mathbf{k}^n$.
\end{defn}

The complexes $\xkn$ and $\kkn$ are called \emph{universal complexes}.

\subsection{Universality}

Our study of the simplicial complexes $\xkn$ and $\kkn$ has been motivated by an important role $K(\Z^n)$ and $K(\F_2^n)$ play in toric topology.

Let $G_d$ be $\Z_2$ when $d=1$ and $S^1$ when $d=2$. Spaces with $G_d^n$-actions  are objects studied in toric topology. Special attention is given to $dn$-manifolds with a locally standard $G_d^n$-action having an $n$-dimensional simple polytope $P^n$ as the orbit space. Under $G_d^n$ for $d=2$ we assume the standard torus $T^n=(S^1)^n$ and for $d=1$ the real torus $(\Z_2)^n$. Those manifolds are usually referred to as small covers for $d=1$ and quasitoric manifolds for $d=2$. They are topological generalizations of toric and real toric varieties. % In \cite{DJ}, Davis and Januszkiewicz proved that each small cover can be obtained from  $Y(\mathbb{Z}_2^n)$ and each quasitoric manifold from $Y(T^n)$ as a pullback via a suitable map.
Recall that  by \cite[Construction~5.12]{MR1897064} the quasitoric manifolds and small covers are determined up to an $G_d^n$-equivariant homeomorphism by a simple polytope $P^n$ and a characteristic map $\lambda\colon \mathcal{F}\rightarrow \pi_{d-1}(G_d^n)$, where $\mathcal{F}$ is the set of facets of $P$. In the case of quasitoric manifolds for each facet $F$ of $P^n$, the vector $\lambda (F)$ is determined up to sign. The characteristic map $\lambda$ satisfies a nondegeneracy condition, that is, $\{\lambda (F_{i_1}), \dots, \lambda (F_{i_k})\}$ is a unimodular set whenever the intersection of facets $F_{i_1}$, $\dots$, $F_{i_k}$ is nonempty. Let $\partial P^*$ be a simplicial complex which is the boundary of the dual to $P$. Then there is a simplicial map $\bar{\lambda}\colon \partial P^*\rightarrow K (\pi_{d-1}(G^n_d))$  defined on the vertices of $\partial P^*$ by $\lambda$ such that a vertex $v$ of $\partial P^*$ dual to the facet $F_v$ of $P$ is sent to the line $l (\lambda (F_v))$. The map $\lambda$ is a characteristic map if and only if $\bar{\lambda}$ is a nondegenerate simplicial map. In \cite[Proposition~2.1]{DJ}, Davis and Januszkiewicz summarized this property in the following proposition:

\begin{prop}
\label{univ}
\begin{itemize}[leftmargin = 8mm]
    \item[(a)] The set of equivalences classes of small covers over $P$ with respect to a $\mathbb{Z}_2^n$-equivariant homeomorphism is in bijection with the set of nondegenerate simplicial maps $f\colon \partial P^*\rightarrow K(\F_2^n)$ modulo natural action of $\mathrm{Aut} (\F_2^n)$.
    \item[(b)] The set of equivalences classes of quasitoric manifolds over $P$ a $T^n$-equivariant homeomorphism is in bijection with the set of nondegenerate simplicial maps $f\colon \partial P^*\rightarrow K(\Z^n)$ modulo natural action of $\mathrm{Aut} (\Z^n)$.
\end{itemize}
\end{prop}

Therefore, the set of equivalences classes of $G_d^n$-manifolds over $P$ is in bijection with the set of nondegenarate simplicial maps $f\colon \partial P^*\rightarrow K(\pi_{d-1}(G^n_d))$ modulo natural action of $\mathrm{Aut} (\pi_{d-1}(G_d^n))$. This universal property of $K(\pi_{d-1}(G^n_d))$ further motivates the combinatorial and topological study of $K(\pi_{d-1}(G^n_d))$.

Using omniorientation, Buchstaber, Panov, and Ray~\cite{MR2337880} constructed a quasitoric representative for every complex cobordism class $\Omega^U$.  A choice of omniorientation on a given quasitoric manifold $M^{2n}$ over $P^n$ is equivalent to a choice of orientation on $P^n$ together with a choice of facet vectors $\lambda (F_{i})$.
More precisely, recall that a pair $(P,\Lambda )$ is called a combinatorial quasitoric pair if $P$ is an oriented combinatorial simple $n$-polytope with $m$ facets and $\Lambda$ is an $n\times m $ integer matrix with the property that for each vertex $v$ appropriate minor $\Lambda_v$ of $\Lambda$, consisting of columns whose facets contain $v$, is invertible. Two combinatorial quasitoric pairs $(P,\Lambda )$ and $(P',\Lambda ')$ are equivalent if $P=P'$ with the same orientation and $ \Lambda = \Phi \Lambda '$ for some $n \times n$ integer matrix $\Phi $ with determinant 1.

Buchstaber, Panov, and Ray~\cite{MR2337880} showed that there is a 1-1 correspondence between equivalence classes of omnioriented quasitoric manifolds and combinatorial quasitoric pairs.
We use this combinatorial description of omnioriented quasitoric manifolds to describe a universal property of $X(\Z^n)$.

\begin{prop}
\label{uni}
There is a 1-1 correspondence between equivalence classes of omnioriented quasitoric manifolds with positive orientation on $P$ and nondegenerate simplicial maps
\[
f\colon\partial P^* \to X(\Z^n)
\]
 modulo natural action of $\mathrm{Aut} (\mathbb{Z}^n)$.
\end{prop}

\begin{proof}
For a combinatorial quasitoric pair $(P,\Lambda )$, where $P$ is an oriented combinatorial simple $n$-polytope with $m$ facets and $\Lambda$ is an $n\times m $ integer matrix with the property that for each vertex $v$ appropriate minor $\Lambda_v$ of $\Lambda$ is invertible, define a map $f\colon\partial P^*\to X(\Z^n)$ on vertices $v_i$ of $\partial P^*$ to be equal to the $i${th} column of $\Lambda$. The condition that $\Lambda_v$ is invertible for each vertex $v$ of $P$ is equivalent to $f$ being nondegenerate.

On the other hand, for a nondegenerate simplicial map $f\colon \partial P^*\to X(\Z^n)$ we construct $\Lambda$ by placing vectors $f(v_i)$ as appropriate columns and thus obtain a combinatorial quasitoric pair $(P,\Lambda)$.
\end{proof}

Note that $X(\F_2^n)\cong K(\F_2^n)$ and therefore, in Proposition~\ref{univ} we may use $\xkn$ instead of $\kkn$.

As a side remark, to avoid possible confusion, let us point out that Davis and Januszkiewicz in~\cite{DJ} denoted by $X(\mathbf{k}^n)$ the barycentric subdivision of simplicial complex of all unimodular subsets of $\k^n$.

\subsection{Basic properties}

The universal complexes $\xkn$ and $\kkn$ are closely connected with each other as we will show. We have the following map
\[
\phi \colon \xkn \to \kkn
\]
where $\phi (v)$ is the line through origin which contains $v$, that is, $\phi (v) = \{ r\cdot v \, | \, r \in \k \}$. Note that $\phi (v)$ is unimodular since $v$ is.

In order to define a map in the other direction, we need to have the ordering on $\k^n$.

In the case $\k = \F_p$ we represent $\F_p = \{0,1,\ldots , p-1\}$.

Now we define the ordering on $\F_p^n$ as follows:
\[
(a_1,\ldots ,a_n)< (b_1,\ldots , b_n) \text{ iff } a_1=b_1, \ldots , a_{k-1}=b_{k-1}, a_k< b_k.
\]
Now we have the map
\[
\xi \colon K(\F_p^n) \to X(\F_p^n)
\]
where $\xi (l)$ is the smallest element from the line $l$. We could actually pick any nonzero element from $l$.

In the case $\k = \Z$ for $\xi (l)$, we take the smallest element from the line $l$ among all elements whose first nonzero coordinate is positive. Note that $\xi (l)$  is a generator of $l$, that is,
\[
l = \{ r\cdot \xi (l) \, | \, r \in \Z \}
\]
and that a line $l$ has two generators $\xi(l)$ and $-\xi (l)$.

We need a criterion to recognize the simplices from $\xkn$ and $\kkn$.
To vectors $u_0,\ldots, u_{k}$ of $\k^n$ we associate the $n\times (k+1)$ matrix $U=\Big( u_0\cdots u_k\Big)$ whose columns contain given vectors $u_i$, $0\leq i\leq k$.

\begin{prop}

\label{GCD}
\begin{itemize}[leftmargin = 8mm]
\item[(a)]
A set $\{u_0,\ldots ,u_k \}$ of vertices from $X(\F_p^n)$ is a $k$-simplex in $X(\F_p^n)$ if and only if  $\rank (U)= k+1$.
\item[(b)]
A set $\{u_0,\ldots ,u_k \}$ of vertices from $X(\Z^n)$ is a $k$-simplex in $X(\Z^n)$ if and only if the greatest common divisor of all $(k+1)\times (k+1)$ minors of the matrix $U$ is equal to one.
\item[(c)]
A set $\{l_0,\ldots ,l_k \}$ of vertices from $\kkn$ is a $k$-simplex in $\kkn$ if and only if $\{ \xi (l_0),\ldots , \xi (l_k) \}$ is a $k$-simplex in $\xkn$.

\end{itemize}
\end{prop}
\begin{proof}
The proof of part (a) follows because $\F_p$ is a field and the fact that an $n \times (k+1)$ matrix can be extended to an invertible $n\times n$ matrix if and only if its rank is $k+1$.

The proof of part (b) is longer and we will give it in detail.

Since the vectors $u_0, u_1, \ldots, u_{k}$ belong to a basis of $\Z^n$, there exist vectors $v_{k+1}, \dots, v_{n-1}$ of $\Z^n$ such that $\det \Big(u_0 \cdots u_k v_{k+1} \cdots v_{n-1}\Big)=\pm 1$. For $J=\{j_0, j_1, \dots, j_k\}\subset\{0, 1, \dots, n-1\}$, let $U_J$  denote the $(k+1)\times (k+1)$ minor formed out of  $j_0$-th, $j_1$-st, $\dots$, $j_{k}$-th rows of the matrix $U=\Big( u_0\cdots  u_k\Big)$.
Let $\rho$ be a permutation of the set $\{i_1, i_2, \dots, i_{n-k-1}\}=\{0, \dots, n-1\}\setminus J$. Then for all $J\subset \{0, 1, \dots, n-1\}$, $|J|=k+1$ and permutations $\rho$ of the set $\{0, 1, \dots, n-1\}\setminus J$, there exist $\varepsilon_{J, \rho}=\pm 1$ such that
\[
\det \Big(u_0 \cdots  u_k v_{k+1} \cdots v_{n-1}\Big)=\]
\[
\sum_{\substack{J \subset \{0, 1, \dots, n-1\}\\|J|=k+1} }\sum_{\rho}\varepsilon_{J, \rho} v_{k+1} (\rho (i_1)) v_{k+2} (\rho (i_2))\cdots v_{n-k}(\rho (i_{n-k}))  U_J =\pm 1.
\]
The claim follows directly from the determinant expansion above,
as the greatest common divisor of all $(k+1)\times (k+1)$ minors of the matrix $U$ must divide $1$.

Now we prove that if the greatest common divisor of all $(k+1)\times (k+1)$ is $1$, the set of vectors $u_0,\ldots, u_{k}$ can be  extended to a basis of $\Z^n$, that is, there is a matrix $ \Big(u_0 \cdots u_k v_{k+1} \cdots v_{n-1}\Big)$ whose determinant is $\pm 1$. Recall that due to the condition on $(k+1)\times (k+1)$ minors, $\gcd (u_i(1), \dots, u_i(n))=1$ for all $0\leq i\leq k$.

Let $A$ be an integer matrix. Let $T_{(i, j;\lambda)}$, $i\neq j$ be a transformation that adds the $j$-th row multiplied by $\lambda$ to the $i$-th row and let $S_{(i, j;\lambda)}$ be a transformation that adds $j$-th column multiplied by $\lambda $ to the $i$-th column. Denote by $\bar{T}_{(i, j;\lambda)}\colon =T_{(i, j;-\lambda)}$  and $\bar{S}_{(i, j;\lambda)}\colon =S_{(i, j;-\lambda)}$. Define the relation $\sim$ on the set of matrices by saying that  $A\sim B$ if matrix $B$ can be obtained from $A$ by applying  transformations $T$ and $S$ finitely many times. The following identities hold
\[
(\bar{T}\circ T)(A)=A \text{ and } (\bar{S}\circ S)(A)=A.
\]
Thus, the relation $\sim$ is an equivalence relation.

\begin{lemma}\label{apply} Suppose that $A$ is a $n\times k$ matrix and let $\gcd (a_{1, 1}, \dots, a_{1, n})=d.$ Then $A\sim B$, where $B$ is a $n\times k$ matrix such that $b_{1, 1}=d$, $b_{1, 2}=\cdots=b_{1, k}=0$. Moreover, if $d=1$ matrix $B$ may also satisfy that $b_{2, 1}=\cdots=b_{n, 1}=0$.
\end{lemma}

\begin{proof}
For $n=2$, the statement follows by the Euclidian algorithm for $\gcd(a_{1, 1}, a_{1, 2})=d$ which determines a sequence of $T$ transformations, which for $d=1$ are followed by a sequence of $S$ transformations. As $\gcd (a_{1, 1}, \dots, a_{1, n})=\gcd (a_{1, 1}, (a_{1, 2} \dots, a_{1, n}))$, the claim for general $n$ is obtained using induction on $n$ and the base case $n=2$.
\end{proof}

After applying $T_{(i, j;\lambda)}$ to $U$, the $(k+1)\times(k+1)$ minor $U'_J$ corresponding to a set $J=\{j_0, j_1, \dots, j_k\}\subset\{0, 1, \dots, n-1\}$ in matrix $T_{(i, j;\lambda)}(U)$ is equal to $U_J$ if $i\not\in J$ or both $i, j\in J$. If $i\in J$ and $j\not\in J$, then $U'_J=U_J \pm \lambda U_{(J\setminus \{i\})\cup \{j\}}$. However, the greatest common divisor of all $(k+1)\times (k+1)$ minors remains $1$ in $T_{(i, j;\lambda)}(U)$.  The same holds for the matrix $S_{(i, j;\lambda)}(U)$ because this matrix preserves all $(k+1)\times (k+1)$ minors.

Since $\gcd (u_1(1), \dots, u_1(n))=1$, by \Cref{apply} we have that $U\sim \Big(e_1 u'_1 \cdots u'_k\Big)$ and the greatest common divisor of all $k\times k$ minors of $\Big( u'_1 \cdots u'_k\Big)$ is equal to $1$. Then $\gcd (u'_1(1), \dots, u'_1(n))=1$ and by applying~\Cref{apply} to $\Big( u'_1 \cdots u'_k\Big)$  we obtain that $U\sim\Big( e_1 e_2 u''_2 \cdots u''_k\Big) $. Continuing this process in the end, we obtain that $U\sim \Big(e_1 e_2 \cdots e_{k+1}\Big)$. Thus, there is a sequence $P_1, \dots, P_t$, where each $P_i$ is some $T$ or $S$ transformation such that
\[
P_t \circ (P_{t-1} \circ (\cdots \circ (P_1))\cdots) (U)=\Big(e_1 e_2, \cdots e_{k+1}\Big).
\]
Then
\[
\bar{P}_1 \circ (\bar{P}_{2} \circ (\cdots \circ (\bar{P}_t))\cdots) \Big(e_1 e_2 \cdots e_{n-1}\Big)=\Big(u_0 u_1 \cdots u_{k} v_{k+1} \cdots v_{n-1}\Big)
\]
for some vectors $v_{k+1}, \dots, v_{n-1}\in \Z^n$. Our claim follows since the transformations $T$ and $S$ preserve the determinant (up to the sign).

For part (c) note that if $\{ l_0,\ldots ,l_k\}$ is a simplex in $\kkn$ then $ \spn (l_0\cup \cdots \cup l_k )$
is a $(k+1)$-dimensional unimodular subspace in $\k^n$ generated by $\{\xi(l_0), \ldots , \xi(l_k)\}$. Therefore, $\{\xi(l_0), \ldots , \xi(l_k)\}$ is a $k$-simplex in $\xkn$. On the other hand, if $\{\xi(l_0), \ldots , \xi(l_k)\}$ is a $k$-simplex in $\xkn$ then by definition $\{l_0,\ldots ,l_k\}$ is a $k$-simplex in $\kkn$.
\end{proof}

We obtain the following corollaries from part (c) and the definition of $\phi$ and $\xi$.

\begin{cor}
\label{nondeg}
The maps $\phi$ and $\xi$ are nondegenerate simplicial maps.
Moreover $\phi \circ \xi = id_{\kkn}$.
\end{cor}

Those two maps give us that in the Proposition~\ref{univ}, we may use $\xkn$ instead of $\kkn$.

We may identify $\kkn$ with the image of $\xi$ and get the following.

\begin{cor}
\label{retract}
The map
 \[
\xi \circ \phi \colon \xkn \to \xi (\kkn)
\]
is a retraction.
\end{cor}

\subsection{\texorpdfstring{$f$}{f}-vectors of \texorpdfstring{$X(\F_p^n)$}{X(Fp)} and \texorpdfstring{$K(\F_p^n)$}{K(Fp)}}

The \textit{f-vector} of an $(n-1)$-dimensional simplicial complex
$\mathcal{K}$ is the integer vector
$$
\textbf{f}
(\mathcal{K})=\left(f_{-1}, f_0, f_1, \dots, f_{n-1}\right)
$$
where
$f_{-1}=1$ and $f_i=f_i (\mathcal{K})$ denotes the number of $i$-faces
of $\mathcal{K}$ for all $i\geq0$.
The \textit{$f$-polynomial} of  an $(n-1)$-dimensional
simplicial complex $\mathcal{K}$ is
\[
\textbf{f} (t)=t^n+f_0
t^{n-1}+\dots+f_{n-1}.
\]
The \textit{Euler characteristic} of a simplicial complex $\mathcal{K}$ is given by
\begin{equation}
\label{euler} \chi (\mathcal{K})=f_0-f_1+\cdots+(-1)^{n-1}
f_{n-1}.
\end{equation}

The $f$-vectors of $X(\F_p^n)$, $K(\F_p^n)$, and the links of their simplices can be explicitly described in terms of $p$ and $n$. The $f$-vectors carry important combinatorial information of $X(\F_p^n)$ and $K(\F_p^n)$, and knowing them helps to determine not only the homotopy type of the universal complexes $X(\F_p^n)$ and $K(\F_p^n)$ but also the Buchstaber invariant of simplicial complexes.

Since the simplicial complexes $X(\F_p^n)$ and $K(\F_p^n)$ are given as the sets of unimodular subsets of $\mathbb F^n_p$, they are clearly  pure, $(n-1)$-dimensional simplicial complexes and come with a transitive and simplicial action of $GL(n, \mathbb F_p)$. Therefore to find the $f$-vectors of $X(\F_p^n)$ and $K(\F_p^n)$ and of the links of their simplices, it is enough to compute the number of linearly independent $i$-tuples in $\mathbb F^n_p$, which is a standard exercise in linear algebra. For example, in the case of $X(\F_p^n)$, the calculation reduces to solving the recurrence relation $(i+2)f_{i+1}(X(\F_p^n))=(p^n-p^{i+1})f_i(X(\F_p^n))$.

\begin{lemma}
\label{month}

\begin{enumerate}[label={(\alph*)},ref={\thelemma~(\alph*)},leftmargin = 8mm]
\item  \label{month1} The $f$-vector of $X (\mathbb{F}_p^n)$ is
given by
\begin{equation*}
 f_{i} (X(\mathbb{F}_p^n))= \frac{(p^n-p^{i}) \cdots (p^n-p^0)}{(i+1)!}
\end{equation*}
for $0\leq i\leq n-1$ and $f_{-1}=1$.

\item  \label{month2}The $f$-vector of $K (\mathbb{F}_p^n)$ is
given by
\begin{equation*}
 f_{i} (K (\mathbb{F}_p^n))= \frac{(p^n-p^{i}) \cdots (p^n-p^0)}{(p-1)^{i+1} (i+1)!}
\end{equation*}
for $0\leq i\leq n-1$ and $f_{-1}=1$.

\item  \label{month3}Let $\sigma $ be an $m$-simplex in
$X(\mathbb{F}_p^n)$. The $f$-vector of
$\Lk_{X(\mathbb{F}_p^n)}(\sigma)$ is given by
\begin{equation*}
f_i(\Lk_{X (\mathbb{F}_p^n)}(\sigma)) =
\frac{(p^n-p^{m+1})\cdots (p^n-p^{m+i+1})}{(i+1)!}
\end{equation*}
for $0\leq i\leq n-m-2$ and $f_{-1}=1$.

\item  \label{month4} Let $\sigma $ be an $m$-simplex in  $K
(\mathbb{F}_p^n)$. The $f$-vector of $\Lk_{K
(\mathbb{F}_p^n)}(\sigma)$ is given by
\begin{equation*}
f_i(\Lk_{K (\mathbb{F}_p^n)}(\sigma)) =
\frac{(p^n-p^{m+1})\cdots (p^n-p^{m+i+1})}{(p-1)^{i+1}(i+1)!}
\end{equation*}
for $0\leq i\leq n-m-2$ and $f_{-1}=1$.
\end{enumerate}
\qed
\end{lemma}

In the next lemma,  the count of minimal nonsimplices of $X(\F_p^n):=\mathcal{X}^n$ is given.

\begin{lemma}
\label{p nms}
Let $n\geq 2$, $1\leq j\leq n$, and let $v_0,\ldots , v_j$ be vertices in $\mathcal{X}^n$.
\begin{itemize}
    \item [(a)]
    The set $A:=\{ v_0,\ldots ,v_j\} $ is a minimal $j$-nonsimplex in $\mathcal{X}^n$ iff $\{v_0,\ldots ,v_{j-1}\}$ is a simplex in $\mathcal{X}^n$ and $v_j=\sum_{t=0}^{j-1}  a_tv_t$, $a_t\in \F_p\setminus \{0\}$.
    \item[(b)] For $j\geq 2$,
    \[
    |\{\text{minimal $j$-nonsimplices of } \mathcal{X}^n\}|=\frac{(p^n-1)(p^n-p)\cdots (p^n-p^{j-1})\cdot (p-1)^{j}}{(j+1)!}.
    \]
    \item[(c)] For $j=1$,
    \[
     |\{\text{minimal $1$-nonsimplices of } \mathcal{X}^n\}|=\frac{(p^n-1)(p-2)}{2}.
    \]
\end{itemize}
\end{lemma}
\begin{proof}
If $A$ is a minimal nonsimplex, then by definition, $\{v_0,\ldots ,v_{j-1}\}$ is a simplex in $\mathcal{X}^n$ and $v_j$ is a linear combination of other vertices or otherwise $A$ would be a simplex in $\mathcal{X}^n$. Thus, $v_j=\sum_{t=0}^{j-1}  a_tv_t$, for some $a_t\in \F_p$. If, for example, $a_i=0$ then $\{v_0,\ldots ,v_{i-1},v_{i+1},\ldots ,v_j\}$ is not a simplex in $\mathcal{X}^n$ which contradicts the fact that $A$ is a minimal nonsimplex.

On the other hand, since $v_j=\sum_{t=0}^{j-1}  a_tv_t$ we get that $A$ is not a simplex in $\mathcal{X}^n$. We have given that $\{v_0,\ldots,v_{j-1}\}$ is a simplex in $\mathcal{X}^n$ and since all $a_t\neq 0$ we conclude that all other faces of $A$ are simplices in $\mathcal{X}^n$.

For part (b), we have that the number of minimal $j$-nonsimplices is equal to the number of choices of a $(j-1)$-simplices multiplied by the number of choices for coefficients $a_t$. We need to divide it by the number of faces of $A$. So we get that the required number is equal to
\[
f_{j-1}(X)\cdot \frac{(p-1)^{j}}{j+1}
\]
and we use Lemma~\ref{month1} to calculate it.

For part (c), note that the number of minimal 1-nonsimplices is equal to $\binom{f_0(\mathcal{X}^n)}{2} - f_1(\mathcal{X}^n)$ and use Lemma~\ref{month1}.
\end{proof}

\subsection{Other properties}

Matroids are combinatorial objects which generalize the notion of linear independence in vector spaces.
\begin{defn}
A finite simplicial complex $M$ is called \emph{matroid} if for any two simplices $\sigma$, $\tau\in M$ such that $\sigma$ has more elements than $\tau$, there is $v\in \sigma\setminus \tau$ such that $\tau\cup \{v\}\in M$.
\end{defn}

This defining property is called the independent set exchange property. Simplices of a matroid are called independent sets.
%Finite matroids are shellable complexes and have the homotopy type of a wedge of spheres, see~\cite{MR1165544} and~\cite[Theorem~1.3]{MR0744856}. For more information on matroid theory, we refer readers to~\cite{MR1207587} and~\cite{MR1226888}.

\begin{prop}
Simplicial complexes $X(\F_p^n)$ and $K(\F_p^n)$ are matroids.
\end{prop}
\begin{proof}
Let $\sigma=\{v_{i_1}, \dots, v_{i_l}\}$ and  $\tau=\{v_{j_1}, \dots, v_{j_m}\}$ be two simplices in $X(\F_p^n)$  such that $l>m$. As the dimension of a simplex in $X(\F_p^n)$ is one less than the dimension of the span of its vertices, there exists $v\in \sigma$ not lying in the span of the vertices in $\tau$. By definition of $X(\F_p^n)$, we have that for such  $v$, the set $\tau \cup \{v\}$ is a simplex in $X(\F_p^n)$. Analogously, it follows readily that $K(\F_p^n)$  also satisfies the independent set exchange property.
 \end{proof}

\begin{cor}
Simplicial complexes $X(\F_p^n)$ and $K(\F_p^n)$ have the homotopy type of a wedge of $(n-1)$-dimensional spheres.
\end{cor}
\begin{proof}

Bj\" orner proved~\cite[Theorem~1.3]{MR0744856} that a pure $n$-dimensional shellable complex $\mathcal{K}$ of at most countable cardinality has the homotopy type of a wedge of $h$ spheres $S^n$, where $h$ is the characteristic of shelling. Finite matroids are shellable complexes and have the homotopy type of a wedge of spheres, see~\cite{MR1165544} and~\cite[Theorem~1.3]{MR0744856}. For more information on matroid theory, we refer readers to~\cite{MR1207587} and~\cite{MR1226888}.
\end{proof}

To determine the exact number of spheres in the wedge we use the $f$-vectors of  $X(\F_p^n)$ and $K(\F_p^n)$.

\begin{prop}

\label{mainKp}
\begin{enumerate}[leftmargin = 8mm]
\item [(a)]
Let $n\geq 1$ be an integer. The simplicial complex $X (\mathbb{F}_p^n)$ is homotopy equivalent to the wedge of $A_n(p)$ spheres
$S^{n-1}$, where
\[
A_n(p) =(-1)^n + \sum_{i=0}^{n-1}(-1)^{n-1-i}\frac{(p^n-p^{i})\cdots (p^n - p^0)}{(i+1)!}.
\]

\item [(b)]  Let $n > 1$ be an integer. The simplicial complex $K(\mathbb{F}^n_p)$ is homotopy equivalent to the wedge of $B_n(p)$ spheres $S^{n-1}$,
where
\[
B_n(p) = (-1)^n+\sum_{i=0}^{n-1} (-1)^{n-1-i}\frac{(p^n-p^{i}) \cdots (p^n-p^0)}{(p-1)^{i+1} \cdot (i+1)!}.
\]
\end{enumerate}

\end{prop}

\begin{proof}
Since $X(\mathbb{F}^n_p)$ and $K(\mathbb{F}^n_p)$  are finite matroids, they are homotopy equivalent to a wedge of $(n-1)$-spheres. We determine the number of spheres in the wedge using the Euler characteristic~\eqref{euler},
\[
\chi(K (\mathbb{F}_p^n))= 1+(-1)^{n-1}B_n(p) = f_0(K (\mathbb{F}_p^n))- \ldots +(-1)^{n-1}f_{n-1}(K (\mathbb{F}_p^n))
\]
which implies that
\begin{equation}
\label{euler2}
B_n(p) = (-1)^n+\sum_{i=0}^{n-1}(-1)^{n-1-i}f_i(K (\mathbb{F}_p^n)).
\end{equation}
Finally, using Lemma~\ref{month2} we obtain
\[
B_n(p) = (-1)^n+\sum_{i=0}^{n-1} (-1)^{n-1-i}\frac{(p^n-p^{i}) \cdots (p^n-p^0)}{(p-1)^{i+1} \cdot (i+1)!}.
\]
The value of $A_n (p)$ follows analogously  from $\chi(X (\mathbb{F}_p^n))= 1+(-1)^{n-1}A_n(p)$.
\end{proof}

The homotopy type of $X(\Z^n)$ and $K(\Z^n)$ is also a wedge of $(n-1)$-spheres.

\begin{prop}
The simplicial complexes $X(\Z^n)$ and $K(\Z^n)$ are homotopy equivalent to a countably infinite wedge of $(n-1)$-spheres, for $n\geq 1$.
\end{prop}
\begin{proof}
With the notation from \cite{Church}, note that for $\mathcal{O}_S = \Z $ we have that $\mathcal{B}_n(\mathcal{O}_S) = X(\Z^n)$ and applying Theorem E from \cite{Church} we get that $X(\Z^n) $, being Cohen-Macaulay in the sense of the  definition from \cite{Church}, is homotopy equivalent to a wedge of $(n-1)$-spheres.

It remains to prove that there are infinitely many spheres in the wedge.
For a top simplex  $\sigma=\{u_1,u_2,\ldots ,u_n\}$ we fix order of its vertices and for $\mu=(\mu_1,\ldots ,\mu_n)\in \{-1,+1\}^n$ we define
\[
\mu (\sigma):= \{\mu_1(u_1),\ldots , \mu_n(u_n)\}.
\]

For each $(n-1)$-simplex $\sigma \in X(\Z^n)$, let
\[
c(\sigma):=\sum_{\mu \in \{-1,\, 1\}^{n}} \mu_1\cdots \mu_n \cdot \mu (\sigma).
\]
One may prove by induction on $n$ that $c(\sigma)$ is a cycle, and because it is a top-dimensional cycle, it represents a nontrivial homology class. Note that $c(\sigma)$'s are  linearly independent with the only relation $c(\sigma )= c(\mu (\sigma ))$.
Since $X(\Z^n)$ has countably infinitely many top simplices, there are countably infinitely many independent generators of $H_{n-1}(X(\Z^n);\Z)$  proving that $X(\Z^n)$ is homotopy equivalent to a countably infinite wedge of $(n-1)$-spheres.

By Corollary~\ref{retract} $K(\Z^n)$ is a retract of $X(\Z^n)$ and therefore, it is an $(n-1)$-dimensional, $(n-2)$-connected simplicial complex and $H_{n-1}(K(\Z^n);\Z)$ is free abelian group. We conclude that $K(\Z^n)$ is a wedge of $(n-1)$-spheres.

Denote by $L_i, i=1,\ldots n$, the lines through $e_i$ and by $L_0$ the line through $(1,1,\ldots ,1)$. Then $\sigma_i := L_0L_1\ldots \hat{L_i}\ldots L_n, i=0,1,\ldots ,n$ is a  $(n-1)$-simplex in $K(\Z^n)$ and
\begin{equation}
\label{top cycle}
c:= \sum_{i=0}^n (-1)^{i}\sigma_i
\end{equation}
is a an $(n-1)$-cycle. Because $c$ is a nonzero top-dimensional cycle, it gives a nonzero homology class. For every automorphism $\xi$ of $\Z^n$ we have that $[\xi (c)]\neq 0$ in $H_{n-1}(K(\Z^n);\Z)$. We pick $\xi_i, i\in \N$ so that simplices $\xi_i(c)$ are disjoint and we get that their homology classes are independent. Therefore, $H_{n-1}(K(\Z^n);\Z)$ has infinitely many generators and the proposition is proved.
\end{proof}

Using the same methods, we get that all links of $k$-simplices from $K(\Z^n)$ and $X(\Z^n)$ are homotopy equivalent to a countably infinite wedge of $(n-k-2)$-spheres. We also conclude that $K(\Z^n)$ and $X(\Z^n)$ are Cohen-Macaulay in the sense of the definition from \cite{Church}.

\section{Tor-algebra of the universal complexes}

In this section, we apply algebraic Morse theory to calculate the Tor-algebra of the universal complexes. Before that, we briefly review definitions and main results on the moment angle complex, its cohomology, and the algebraic Morse theory principles. For a  more substantial introduction to the moment angle complexes and the Tor-algebra, we refer the reader to \cite{BPbook}, and for more about the algebraic Morse theory, we refer the reader to \cite{Skol06}.

\subsection{Moment angle complex \texorpdfstring{$\mathcal{Z}_\mathcal{K}$}{ZK}}
Let $\mathcal{K}$ be an $n$-dimensional simplicial complex with the set of vertices $[m]$. Let
us denote by $(\underline{X}, \underline{A})=\{(X_i,
A_i)\}^m_{i=1}$ a collection of topological pairs of CW-complexes.
The polyhedral $\mathcal{K}$ product is a topological space $\mathcal{Z}_\mathcal{K}
(\underline{X}, \underline{A})=\bigcup_{\sigma \in \mathcal{K}} D(\sigma)$
where
$$D(\sigma)=\prod_{i=1}^m Y_i, \qquad\mbox{and}\qquad
Y_i=
\begin{cases}
X_i &\text{if } i\in\sigma,\\
A_i &\text{if } i\not\in\sigma.
\end{cases}
$$
By definition we have $D(\emptyset)=A_1\times \dots\times A_m$.

In the case $(\underline{X}, \underline{A})=\{(D^2,
S^1)\}^m_{i=1}$, $\mathcal{Z}_\mathcal{K} (\underline{X}, \underline{A})$ is denoted by $\mathcal{Z}_\mathcal{K}$ and it is called \textit{the moment angle complex} of $\mathcal{K}$ and in the case $(\underline{X}, \underline{A})=\{(D^1, S^0)\}^m_{i=1}$, $\mathcal{Z}_\mathcal{K} (\underline{X}, \underline{A})$ is denoted by $_{\mathbb{R}} \mathcal{Z}_\mathcal{K}$ and it is called \textit{the real moment angle complex} of $\mathcal{K}$.

They are of special interest in toric topology because they are important in the understanding of combinatorics of $\mathcal{K}$. If $\mathcal{K}$ is a triangulation of a sphere, $\mathcal{Z}_\mathcal{K}$ is a topological $(m+n)$-dimensional manifold.

Let $\k [v_1, \dots, v_m]$ be the polynomial algebra with grading $\deg v_i=2$. Let $\Lambda [u_1, \dots, u_m]$ be the exterior algebra with grading $\deg u_i=1$. Given a subset $I=\{i_1, \dots, i_k\}\subset [m]$ let $v_I$ denote the square-free monomial $v_{i_1}\cdots v_{i_k}$ in $\k [v_1, \dots, v_m]$.
The Stanley-Reisner ring of the complex $\mathcal K$ is the quotient ring $\k[\mathcal{K}]=\k[v_1,\ldots,v_m]/{\mathcal{I}_\mathcal{K}}$, where
$\mathcal{I}_\mathcal{K}=(v_I:I\not\in\mathcal{K})$.
The quotient algebra is
\[
R^*(\mathcal{K})=\Lambda[u_1,\ldots,u_m]\otimes \k[\mathcal{K}]/(v_i^2=u_iv_i=0,1\le i\le m)
\]
where the differential is defined by
$$
\partial(u_{\{a_1,\ldots, a_i\}} v_B)=\sum_{\substack{{1\le k\le i}\\ {\{a_k\}\cup B\text{ is a simplex}}}} (-1)^k u_{\{a_1,\ldots, a_{k-1},a_{k+1},\ldots a_i\}} v_{\{a_k\}\cup B}
$$
and has bigratuation where the degree of $u_Av_B\in R^*(\mathcal{K})$ is $\deg(u_Av_B)=(-|A|,2(|A|+|B|))$; i.e.,
$\deg u_i=(-1,2)$ and $\deg v_i=(0,2)$. Its homology is called Tor-module of $\mathcal K$ and denoted by $\Tor_{\k[v_1,\ldots,v_m]}(\k[\mathcal{K}],\k)$.

The cohomology ring of $\mathcal{Z}_K$ over $\k$ was obtained by Buchstaber, Panov, and Baskakov.

\begin{thm}\label{zkcoh} There are isomorphisms, functorial in $\mathcal{K}$, of bigraded algebras \begin{equation*}H^{\ast, \ast} \left (\mathcal{Z}_\mathcal{K}; \k \right)\cong
\Tor_{\k [v_1, \dots, v_m]} (\k[\mathcal{K}],
\k)\cong H^{\ast, \ast}[\Lambda [u_1, \dots, u_m]\otimes
\k[\mathcal{K}], d],\end{equation*} where the bigrading and
the differential on the right hand side are defined by $$\deg\,u_i = (-1, 2), \deg\,v_i = (0, 2), d u_i=v_i, d v_i=0.$$
\end{thm}

Additive structure of $H^*(\mathcal{Z}_\mathcal{K};\k)$ can thus be obtained using a well-known result from combinatorial commutative algebra, the Hochster's formula, which represents the above Tor-algebra as a direct sum of reduced simplicial cohomology groups of all full subcomplexes in $\mathcal{K}$. Recall that for $J\subset [m]$ the simplicial complex $\mathcal{K}_J$ is the full subcomplex of $K$ on the vertex set $J$. Multiplication in $H^*(\mathcal{Z}_\mathcal{K};\k)$ was firstly described by Baskakov in \cite[Theorem~1]{Baskakov02}.

\begin{thm}\label{hoch}
There are isomorphisms of
$\k$-modules
\begin{align*}
\Tor^{-i,2j}_{\k[v_{1},\ldots,v_{m}]}(\k[\mathcal{K}];\k) &\cong  \bigoplus_{J\subset [m], |J|=j}
\widetilde{H}^{|J|-i-1} (\mathcal{K}_J;\k),\\
H^{l}\left (\mathcal{Z}_\mathcal{K}; \k \right ) &\cong  \bigoplus_{J\subset [m]}
\widetilde{H}^{l-|J|-1} (\mathcal{K}_J;\k).
\end{align*}
These isomorphisms sum up into a ring isomorphism \begin{align*}
H^{\ast}\left (\mathcal{Z}_\mathcal{K}; \k \right)\cong \bigoplus_{J\subset [m]}
\widetilde{H}^{\ast} (\mathcal{K}_J;\k),
\end{align*}
where the ring structure on the
right hand side is given by the canonical maps $$H^{k-|I|-1}
(\mathcal{K}_I;\k)\otimes H^{l-|J|-1} (\mathcal{K}_I;\k)\rightarrow H^{k+l-|I|-|J|-1}
(\mathcal{K}_{I\cup J};\k)$$ which are induced by simplicial maps $\mathcal{K}_{I\cup
J}\rightarrow \mathcal{K}_I \ast \mathcal{K}_J$ for $I\cap J=\emptyset$ and zero
otherwise.
\end{thm}

\subsection{Algebraic Morse theory} Suppose we are given a finitely generated cochain complex of free $R$-modules $\mathcal{C}_*=(\mathcal{C}_k,\partial_k)_{k\in\Z}$. Then $\mathcal{C}_k=\oplus _{v\in I_k} \mathcal{C}_{k,v}$ for some finite index set $I_k$, where $\mathcal{C}_{k,v}\cong R$ for all $k\in \Z$ and all $v\in I_k$. For $u\in I_k$ and $v\in I_{k+1}$, let $\partial_{u,v}$ denote the $R$-module morphism $\mathcal{C}_{k,u}\smash{\overset{\iota}{\longrightarrow}} \mathcal{C}_k\smash{\overset{\partial}{\longrightarrow}} \mathcal{C}_{k+1}\smash{\overset{\pi}{\longrightarrow}} \mathcal{C}_{k+1,v}$, where $\iota$ is the inclusion and $\pi$ the projection. The associated \emph{digraph of $\mathcal{C}_\ast$}, denoted by $\Gamma_{\mathcal{C}_\ast}$, is a directed simple graph with vertex set the disjoint union $\bigsqcup_{k\in\Z} I_k$ and edge set $\cup_{k\in\Z}\{(u,v)\mid u\in I_{k-1},j\in I_{k},\partial_{u,v}\neq 0\}$. We denote each edge $(u,v)$ by $u \to v$.

Let $\mathcal{M}$ be a subset of the set of edges of $\Gamma_{\mathcal{C}_*}$. The associated \emph{digraph of $\mathcal{M}$}, denoted $\Gamma_{\mathcal{C}_*}^\mathcal{M}$, is the digraph $\Gamma_{\mathcal{C}_*}$ with all the edges in $\mathcal{M}$ having reversed orientation,
$$
\big(\{\text{vertices of }\Gamma_{\mathcal{C}_*}\},\{u\to v\mid \,u\to v\in\Gamma_{\mathcal{C}_*}\setminus \mathcal{M}\}\cup\{v\to u\mid \,u\to v\in\mathcal{M}\}\big).
$$
We say that $\mathcal{M}$ is a \emph{Morse matching} on $\Gamma_{\mathcal{C}_*}$ when the following conditions hold:
\begin{enumerate}
\item edges in $\mathcal{M}$ have no common endpoints, i.e., whenever $u\to v,u' \to v' \in \mathcal{M}$ then $\{u,v\}\cap\{u',v'\}=\emptyset$,
\item for every edge $u\to v\in\mathcal{M}$, the corresponding map $\partial_{u,v}$ is an isomorphism,
\item $\Gamma_{\mathcal{C}_*}^\mathcal{M}$ contains no directed cycles.
\end{enumerate}
A vertex in $\Gamma_{\mathcal{C}_*}^\mathcal{M}$ that is not incident to an edge in $\mathcal{M}$ is called \emph{$\mathcal{M}$-critical}. We denote $\mathring{\mathcal{M}} =\{\mathcal{M}$-critical vertices of $\Gamma_{\mathcal{C}_*}^\mathcal{M}\}$
and $\mathring{\mathcal{C}}_k=\bigoplus_{v\in I_k\cap \mathring{\mathcal{M}}}\mathcal{C}_{k,v}$.
If $u \in I_{k-1}\cap \mathring{\mathcal{M}}$ and $v\in I_k\cap \mathring{\mathcal{M}}$, then the set of all directed trails in $\Gamma_{\mathcal{C}_*}^\mathcal{M}$ from $u$ to $v$ with vertices in $I_{k-1}\cup I_{k}$ is denoted by $\Gamma_{u,v}^\mathcal{M}$. By $(1)$ and $(3)$, such trails are `zig-zag' and obviously finite paths. If $\gamma = (u=v_1\to\ldots\to v_{2i}=v) \in \Gamma_{u,v}^\mathcal{M}$, then the \emph{gradient path} along $\gamma$ is the signed morphism that maps from $\mathcal{C}_{k-1,u}$ to $\mathcal{C}_{k,v}$ along the path $\gamma$, i.e.
$$
\partial_\gamma=(-1)^{i-1}\partial_{v_{2i-1}v_{2i}}\partial^{-1}_{v_{2i-1}v_{2i-2}}
\ldots\partial^{-1}_{v_5v_4}\partial_{v_3v_4}\partial^{-1}_{v_3v_2}\partial_{v_1v_2}.
$$
The inverses exist by $(2)$. Define $\mathring{\partial}\colon \mathring{\mathcal{C}}_{k-1}\rightarrow\mathring{\mathcal{C}}_{k}$ on
a generator $u\in I_{k-1}\cap \mathring{\mathcal{M}}$ by
$$
\mathring{\partial}(u) = \textstyle{\sum_{v\in \mathring{\mathcal{M}}_{k},\gamma\in\Gamma_{u,v}^\mathcal{M}}\partial_\gamma(u)}.
$$

\begin{thm}[\cite{Joll05},\cite{Skol06}]\label{JolSkol}
If $\mathcal{M}$ is a Morse matching, then cochain complex $(\mathcal{C}_k,\partial_k)_{k\in\Z}$ is homotopy equivalent to the cochain complex $(\mathring{\mathcal{C}}_k,\mathring{\partial}_k)_{k\in\Z}$.
\end{thm}\vspace{1mm}

\subsection{Bigraded Betti numbers of the universal complexes}
We will use the Morse theory for a calculation of the Tor-algebra of a matroid $\mathcal{K}$. Recall that all maximal simplices in a matroid $\mathcal K$ have the same dimension, and the {\emph rank} of a matroid is the number of vertices in a maximal simplex.

Let us define a Morse matching on $R^*(\mathcal{K})$.
For every $u_Av_B\in R^*(\mathcal{K})$ we define the sets
\begin{align*}
M(u_Av_B)=&\{a\in A\mid \{a\}\cup B\in \mathcal{K}, \{x\}\cup B\not\in \mathcal{K}\text{ for all }x\in A\cap[a-1], \\
 &\text{ for every }y\in B\cap[a-1]\text{ there is }x\in A\cap[y-1]\text{ such that } \{x\}\cup B\setminus\{y\}\in\mathcal{K}
  \},\\
N(u_Av_B)=&\{b\in B\mid
 \{x\}\cup B\setminus\{b\}\not\in \mathcal{K}\text{ for all }x\in A\cap [b-1],\\
 &\text{ for every }y\in B\cap[b-1]\text{ there is }x\in A\cap[y-1]\text{ such that } \{x\}\cup B\setminus\{y\}\in\mathcal{K}
\}.
\end{align*}
By the definition of the above sets, we have
\begin{itemize}
\item $|M(u_Av_B)|+|N(u_Av_B)|\le 1$,
\item if $M(u_Av_B)=\{a\}$ then $N(u_{A\setminus\{a\}}v_{\{a\}\cup B})=\{a\}$ (and $M(u_{A\setminus\{a\}}v_{\{a\}\cup B})=\emptyset$),
\item if $N(u_Av_B)=\{b\}$ then $M(u_{A\cup\{b\}}v_{B\setminus\{b\}})=\{b\}$ (and $N(u_{A\cup\{b\}}v_{B\setminus\{b\}})=\emptyset$).
\end{itemize}

From the above properties, we have that the set
\begin{align*}
\mathcal{M}=&\{u_A v_B\to u_{A\setminus\{a\}}v_{\{a\}\cup B}\mid a\in M(u_Av_B)\}
\end{align*}
equals to the set $\{u_{\{b\}\cup A} v_{B\setminus\{b\}}\to u_A v_B\mid b\in N(u_Av_B)\}$ and it
satisfies the conditions (1) and (2) from the definition of Morse matching. Let us prove that $\mathcal{M}$ also satisfies the condition (3). Suppose
\begin{align*}
u_{A\setminus\{a\}}v_{\{a\}\cup B}\stackrel{\mathcal M}{\longrightarrow} u_{A}v_{B}\longrightarrow u_{A\setminus\{x\}}v_{\{x\}\cup B}\stackrel{\mathcal M}{\longrightarrow}u_{\{b\}\cup A\setminus\{x\}}v_{\{x\}\cup B\setminus\{b\}}
\end{align*}
is a part of the directed graph $\Gamma_{\mathcal{R}^*(\mathcal{K})}^\mathcal{M}$. We have $M(u_Av_B)=\{a\}$ and $M(u_{\{b\}\cup A\setminus\{x\}}v_{\{x\}\cup B\setminus\{b\}})=\{b\}$.

Suppose $b<a$. By the definition of $a\in M(u_A v_B)$ there exists $c\in A\cap [b-1]$ such that $\{c\}\cup B\setminus\{b\}$ is a simplex. Note that by the definition of the element $a>c$ in $M(u_Av_B)$, the set $\{c\}\cup B$ is not a simplex in $\mathcal K$. Since $u_{A\setminus\{x\}}v_{\{x\}\cup B}$ is in $R^*(\mathcal{K})$ the set $\{x\}\cup B$ is a simplex in $\mathcal K$ and it has more vertices than the simplex $\{c\}\cup B\setminus\{b\}$. Since $\mathcal K$ is a matroid there is an element $y\in \{x\}\cup B$ such that $\{y,c\}\cup B\setminus\{b\}$ is a simplex in $\mathcal K$. There are only two possibilities for $y$, namely $(\{x\}\cup B)\setminus (\{c\}\cup B\setminus\{b\})=\{b,x\}$. Since $\{b\}\cup(\{c\}\cup B\setminus\{b\})=\{c\}\cup B$ is not a simplex, $y=x$ and $\{c,x\}\cup B\setminus\{b\}$ is a simplex. We get that for $c\in A\setminus\{x\}\cap [b-1]$ the set $\{c,x\}\cup B\setminus\{b\}$ is a simplex, hence $b\not\in N(u_{A\setminus\{x\}}v_{\{x\}\cup B})$. Since  the arrow $u_{\{b\}\cup A\setminus\{x\}}v_{\{x\}\cup B\setminus\{b\}}\to u_{A\setminus\{x\}}v_{\{x\}\cup B}$ is in $\mathcal M$, we have $\{b\}=M(u_{\{b\}\cup A\setminus\{x\}}v_{\{x\}\cup B\setminus\{b\}})=N(u_{A\setminus\{x\}}v_{\{x\}\cup B})\ne\{b\}$ which is a contradiction.

We proved that $a<b$. Hence for a cycle
\begin{align*}
u_{A_0}v_{B_0}\stackrel{\mathcal M}{\longrightarrow} u_{A_1}v_{B_1}\longrightarrow u_{A_2}v_{B_2}\stackrel{\mathcal M}{\longrightarrow}u_{A_3}v_{B_3}\longrightarrow\cdots \stackrel{\mathcal M}{\longrightarrow} u_{A_{2k-1}}v_{B_{2k-1}}\longrightarrow u_{A_{2k}}v_{B_{2k}}=u_{A_0}v_{B_0}
\end{align*}
in $\Gamma_{\mathcal{R}^*(\mathcal{K})}^\mathcal{M}$ we have $M(u_{A_0}v_{B_0})<M(u_{A_2}v_{B_2})<\ldots<M(u_{A_{2k}}v_{B_{2k}})=M(u_{A_0}v_{B_0})$
which is impossible. We proved that $\mathcal M$ also satisfies condition (3) of the definition of  Morse matching.

Let us show that $\mathring{\partial}$ is trivial. For every subset of vertices $M\subseteq [m]$ let $R_M^*(\mathcal{K})$ be the subalgebra of $R^*(\mathcal{K})$ generated by those $u_Av_B$ for which $A\cup B=M$. The set of generators $M$ is invariant with respect to the differential, so $R_M^*(\mathcal{K})$ has differential and $R^*(\mathcal{K})=\oplus_{M\subseteq [m]} R_M^*(\mathcal{K})$.
The chain complex $R_M^*(\mathcal{K})$ may be nontrivial only in a bidegree $(-i,2|M|)$.
Let us show that the homology of $R_M^*(\mathcal{K})$ is nontrivial only in the bidegree $(r(\mathcal{K}_M)-|M|,2|M|)$, where $r(\mathcal{K}_M)$ is the rank of the matroid $\mathcal{K}_M$ which is the full subcomplex of $\mathcal{K}$ spanned on a subset $M$.

Let $u_Av_B$ be such that $B$ is not a maximal simplex in $\mathcal{K}_M$.  Let $a\in A$ be the smallest number such that $\{a\}\cup B$ is a simplex. If there is no $b\in B\cap [a-1]$ such that $\{x\}\cup B\setminus\{b\}$ is a simplex for some $x\in A\cap [b-1]$, then $a\in M(u_A v_B)$ otherwise $b\in N(u_A v_B)$ where $b$ is the smallest number in $B\cap [a-1]$ such that there is $x\in A\cap [b-1]$ such that $\{x\}\cup B\setminus\{b\}$ is a simplex. Hence $M(u_A v_B)\ne \emptyset$ or $N(u_A v_B)\ne \emptyset$, therefore $u_Av_B\not\in\mathring{\mathcal{M}}$.
We proved that if $u_Av_B\in\mathring{\mathcal{M}}$ then $B$ is a maximal simplex in $\mathcal M$, so there is no arrow in $\Gamma_{\mathcal{R}^*(\mathcal{K})}$ starting in $u_Av_B$, hence $\mathring\partial(u_A v_B)=0$ and the morphism $\mathring\partial$ is trivial. We proved the following theorem.

\begin{thm}\label{Morsecomplex}
For a matroid $\mathcal{K}$ the Tor-module $\Tor_{\k[v_1,\ldots,v_m]}(\k[\mathcal{K}];\k)$ is isomorphic to the free complex $\mathring{C}_*$ generated by $\mathring{\mathcal M}$.
\end{thm}

\begin{cor}
For a matroid $\mathcal{K}$ the Tor-module $\Tor_{\Z[v_1,\ldots,v_m]}(\Z[\mathcal{K}];\Z)$ has no torsion.
\end{cor}

\begin{rem} The fact that the Tor-module of a matroid $\mathcal {K} $ is torsion free also follows from the Hocshter formula and the observation that a full subcomplex of a matroid is also a matroid.
\end{rem}

From the above Morse matching, we will derive some recursion formulas for Betti numbers of cohomologies $H^{*,*}(X(\F_p^n))$.
The fact that $u_Av_B$ can be in $\mathring{\mathcal M}$ only in the case when $B$ is a maximal simplex in $\mathcal{K}_M$, $M=A\cup B$, implies that $H^{-i,2j}(R_M^*(\mathcal{K}))= 0$ for all $(-i,2j)\ne (r(\mathcal{K}_M)-|M|, 2|M|)$. Since the Euler characteristics of the homology and the chain complex coincide, we have
\begin{align*}
\beta^{r(\mathcal{K}_M)-|M|, 2|M|}(R_M^*(\mathcal{K}))
=\chi(R_M^*(\mathcal{K}))=\sum_{i=0}^{|M|-r(\mathcal{K}_M)}(-1)^i\dim_{\k} R_M^{-i,2|M|}(\mathcal{K}).
\end{align*}
Using the decomposition $R^*(\mathcal{K})=\oplus_{M\subseteq [m]}R_M^*(\mathcal{K})$ we get
\begin{align*}
\beta^{-i, 2j}(R^*(\mathcal{K}))=\sum_{\substack{|M|=j\\ r(\mathcal{K}_M)=j-i}}\sum_{l=0}^{i}(-1)^l\dim_{\k} R_M^{-l,2|M|}(\mathcal{K}).
\end{align*}

In the case of a matroid $X(\F_p^n)=:\mathcal{X}^n$ we can use the following decomposition $R^*(\mathcal{X}^n)=S^*(\mathcal{X}^n)\oplus T^*(\mathcal{X}^n)$, where $S^*(\mathcal{X}^n)$ is the subalgebra of $R^*(\mathcal{X}^n)$ generated by those $u_Av_B$ for which $\spn(A\cup B)=\F_p^n$ and $T^*(\mathcal{X}^n)$ is the subalgebra of $R^*(\mathcal{X}^n)$ generated by those $u_Av_B$ for which $\spn(A\cup B)\ne \F_p^n$.
If $\dim \spn(A\cup B)=l$ there is the unique $l$-dimensional subspace $M$ of $\F_p^n$ such that $A\cup B\subseteq M$.
The number of $l$-dimensional subspaces in $\F_p^n$ equals to the $p$-binomial coefficient ${{n\brack l}}_p=\tfrac{(p)_n}{(p)_l (p)_{n-l}}$, where $(p)_k=(p-1)\ldots (p^k-1)$ (see \cite[p. 20]{GasperMizan}). Hence
\begin{align}\label{recursionforXn}
R^*(\mathcal{X}^n)\cong S^*(\mathcal{X}^n)\oplus\left(\bigoplus_{l=0}^{n-1} {n\brack l}_p S^*(\mathcal{X}^l)\right).
\end{align}
We can summarize the above calculations in the result.

\begin{thm}\label{Propositionrecursion}
For every prime number $p$ we have:
\begin{enumerate}
\item[(a)] If $j-i<n$ then
$\displaystyle{
\beta^{-i,2j}(X(\F_p^n))={n\brack j-i}_p \beta^{-i,2j}(X(\F_p^{j-i}))}$.
\item[(b)] If $j-i=n$ then
$\displaystyle{\beta^{-i,2j}(X(\F_p^n))=\sum_{k=0}^n \frac{(-1)^{n-k}p^{\frac 1 2 k(k-1)}(p)_n}{k!(p)_{n-k}} \binom{p^n-1-k}{j-k}
 -\sum_{l=0}^{n-1} (-1)^{n-l} \beta^{-j+l, 2j}(X(\F_p^n))
}$.
\end{enumerate}
\end{thm}

\begin{proof}
\begin{enumerate}
\item[(a)]  By the above \Cref{Morsecomplex} we have $\beta^{-i,2j}(S^*(\mathcal{X}^l))=0$ for $j-i\ne l$. Using the equation \ref{recursionforXn}, we have
\begin{align*}
\beta^{-i,2j}(X(\F_p^n))&=\beta^{-i,2j}(R^*(\mathcal{X}^n))=\beta^{-i,2j}(S^*(\mathcal{X}^n))+\sum_{l=0}^{n-1}{n\brack l}_p \beta^{-i,2j}(S^*(\mathcal{X}^l))\\
&={n\brack j-i}_p \beta^{-i,2j}(S^*(\mathcal{X}^{j-i}))
={n\brack j-i}_p \beta^{-i,2j}(X(\F_p^{j-i})).
\end{align*}
\item[(b)]
In the case $j-i=n$ we use the equality (\ref{recursionforXn}) and we get
$$
\chi(R^*(\mathcal{X}^n))=\chi(S^*(\mathcal{X}^n))+\sum_{l=0}^{n-1} {n\brack l}_p\chi(S^*(\mathcal{X}^l))
$$
The number of generators $u_Av_B$ in $R^{-j+k,2j}(\mathcal{X}^n)$ equals
\begin{align*}
\frac{(p^n-1)(p^n-p)\cdots(p^n-p^{k-1})}{k!}\binom{p^n-1-k}{j-k}=
\frac{p^{\frac 1 2 k(k-1)}(p)_n}{k!(p)_{n-k}} \binom{p^n-1-k}{j-k}.
\end{align*}
Let $l=0,\ldots,n$. If $k<j-l$ then $S^{-k,2j}(\mathcal{X}^l)=\emptyset$, since there is no simplex in $\mathcal{X}^l$ with $j-k>l$ vertices.
Hence the complex $S^{*,2j}(\mathcal{X}^l)=\oplus_{k=i}^{j}S^{-k,2j}(\mathcal{X}^l)$ is a chain subcomplex of $R^{*,*}(\mathcal{X}^n)$ and
by \Cref{Morsecomplex} we have $\chi(S^{*,2j}(\mathcal{X}^l))=(-1)^{j-l}\beta^{-i, 2j}(R^{*,*}(\mathcal{X}^l))$. Hence
\begin{align*}
\beta^{-i, 2j}(X(\F_p^n))
&=(-1)^i\left(\chi(R^{*,2j}(\mathcal{X}^n))-\sum_{l=0}^{n-1} {n\brack l}_p\chi(S^{*,2j}(\mathcal{X}^l))\right)\\
&=(-1)^i\chi(R^{*,2j}(\mathcal{X}^n))-(-1)^i\sum_{l=0}^{n-1} {n\brack l}_p(-1)^{j-l}\beta^{-j+l,2j}(R^{*,*}(X^l))\\
&=(-1)^i\chi(R^{*,2j}(\mathcal{X}^n))-\sum_{l=0}^{n-1} (-1)^{n-l}{n\brack l}_p\beta^{-j+l,2j}(R^{*,*}(X^l)).
\end{align*}
\end{enumerate}
The theorem follows from the fact that $\beta^{-j+l, 2j}(X(\F_p^n))={n\brack l}_p\beta^{-j+l,2j}(R^{*,*}(X^l))$ which follows from \Cref{Morsecomplex}.
\end{proof}

Using the above theorem for the prime $p=2$, it is easy to calculate Betti numbers of the universal complex $X(\F_2^n)$ for small numbers $n$. For $n=2$ beside $\beta^{0,0}(X(\F_2^2))$ there is only one nonzero Betti number, namely $\beta^{-1,6}(X(\F_2^2))=1$. For the cases $n=3,4$ see \Cref{BettiX3} and \Cref{BettiX4}.
\begin{center}
\begin{table}[h]
\begin{tabular}{ |c|c c c c c c c| }
\hline
$l$\textbackslash$i$ & 1 & 2 & 3 & 4 & 5 & 6 & 7\\
\hline
1 & 0 & 0 & 0 & 0 & 0 & 0 & 0\\
2 & 0 & 0 & 7 & 0 & 0 & 0 & 0\\
3 & 0 & 0 & 0 & 7 &42 &42 &13 \\
\hline
\end{tabular}
\caption{The table of the Betti numbers $\beta^{l-i,2i}(X(\F_2^3))$.} \label{BettiX3}
\end{table}
\end{center}
{\scriptsize
\begin{center}
\begin{table}[h]
\begin{tabular}{ |c|c c c c c c c c c c c c c c c| }
\hline
$l$\textbackslash$i$ & 1 & 2 & 3 & 4 & 5 & 6 & 7 & 8 & 9 & 10 & 11 & 12 & 13 & 14 & 15\\
\hline
1 & 0 & 0 & 0 & 0 & 0 & 0 & 0 & 0 & 0 & 0 & 0 & 0 & 0 & 0 & 0\\
2 & 0 & 0 & 35 & 0 & 0 & 0 & 0 & 0 & 0 & 0 & 0 & 0 & 0 & 0 & 0\\
3 & 0 & 0 & 0 & 105 & 630 & 630 & 195 & 0 & 0 & 0 & 0 & 0 & 0 & 0 & 0\\
4 & 0 & 0 & 0 & 0 & 168 & 4480 & 27420 & 79695 & 140140 & 163548 &
 130725 & 71225 & 25410 & 5370 & 511\\
\hline
\end{tabular}
\caption{The table of the Betti numbers $\beta^{l-i,2i}(X(\F_2^4))$.}\label{BettiX4}
\end{table}
\end{center}
}
\begin{rem}
Using the first part of \Cref{Propositionrecursion}, we get that the second row in the \Cref{BettiX4} is just the $35$ times the second row in the table for Betti numbers of $X(\F_2^2)$, since there are $35={4\brack 2}_2$ planes in $\F_2^4$, and the third row in the \Cref{BettiX4} is 15 times of the third row of \Cref{BettiX3} since there are $15={4\brack 3}_2$ 3-dimensional subspaces in $\F_2^4$. The last row of \Cref{BettiX4} is obtained by the first three rows and the second part of \Cref{Propositionrecursion}.

Similar tables can be obtained for odd primes $p$, but tables are huge, for examples the table of Betti numbers of $X(\F_3^4)$ has $3^4-1=80$ columns, and the largest Betti number is $\beta^{-4,84}(X(\F_3^4))\approx 6\cdot 10^{27}$.
\end{rem}

In the case of a matroid $K(\F_p^n)=:\mathcal{K}^n$ we can use an analog decomposition $R^*(\mathcal{K}^n)=S^*(\mathcal{K}^n)\oplus T^*(\mathcal{K}^n)$ as in the case $X(\F_p^n)$.
Again for $j-i<n$ we get
\begin{equation*}
\beta^{-i,2j}(K(\F_p^n))={n\brack j-i}_p \beta^{-i,2j}(K(\F_p^{j-i})).
\end{equation*}
The number of vertices in $K(\F_p^n)$ is $m=\tfrac{p^n-1}{p-1}$ and the number of generators $u_Av_B$ in $R^{-i,2j}(\mathcal{K}_n)$ equals
\begin{align*}
\frac{m(m-1)(m-{2\brack 1}_p)\cdots(m-{j-i-1\brack 1}_p)}{(j-i)!}&\binom{m-j+i}{i}\\
&\hspace{-2cm}=\frac{(p^n-1)(p^n-p)\cdots(p^n-p^{j-i-1})}{(j-i)!(p-1)^{j-i}}\binom{m-j+i}{i}\\
&\hspace{-2cm}=\frac{p^{\frac 1 2 (j-i)(j-i-1)}(p)_n}{(j-i)!(p-1)^{j-i}(p)_{n-j+i}} \binom{m-j+i}{i}.
\end{align*}
Hence
\begin{align*}
\beta^{-i, 2j}(K(\F_p^n))
=\sum_{k=0}^n \frac{(-1)^{n-k}p^{\frac 1 2 k(k-1)}(p)_n}{k!(p-1)^k(p)_{n-k}} \binom{m-k}{j-k}
 -\sum_{l=0}^{n-1} (-1)^{n-l} \beta^{-j+l, 2j}(K(\F_p^n)).
\end{align*}

Let us summarize the above calculations.
\begin{thm}\label{mt2} For every prime number $p$ we have:
\begin{enumerate}
\item[(a)] If $j-i<n$ then $\displaystyle{\beta^{-i,2j}(K(\F_p^n))={n\brack j-i}_p \beta^{-i,2j}(K(\F_p^{j-i}))}$.
\item[(b)] If $j-i=n$ then
$\displaystyle{\beta^{-i, 2j}(K(\F_p^n))
=\sum_{k=0}^n \frac{(-1)^{n-k}p^{\frac 1 2 k(k-1)}(p)_n}{k!(p-1)^k(p)_{n-k}} \binom{m-k}{j-k}
 -\sum_{l=0}^{n-1} (-1)^{n-l} \beta^{-j+l, 2j}(K(\F_p^n))}$.
\end{enumerate}
\end{thm}

Although Betti numbers of $K(\F_p^n)$ are smaller than Betti numbers of $X(\F_p^n)$ they are still large. For example, the largest Betti number  of $K(\F_3^4)$ is $\beta^{-4,44}(K(\F_3^4))\approx 4.3\cdot 10^{14}$.

\section{Products in the cohomology of \texorpdfstring{$\mathcal{Z}_{X(\F_p^n)}$}{ZX(Fp)} and \texorpdfstring{$\mathcal{Z}_{K(\F_p^n)}$}{ZK(Fp)}}

The cohomology classes of  $\mathcal{Z}_{X(\F_p^n)}$ are represented by the cohomology classes of the full subcomplexes of $X(\F_p^n)$ and Baskakov's formula \cite{Baskakov02} may be used for determination of the cup product in $\widetilde{H}^{\ast}(\mathcal{Z}_{X(\F_p^n)})$. Let us fix $n$ and denote $\mathcal{X}:=X(\F_p^n)$ and $\mathcal{K}:=K(\F_p^n)$. Recall that $X(\F_2^n)=K(\F_2^n)$. For a subset $I$ of the set of vertices of $\mathcal{X}$, let $\mathcal{X}_I$ denotes the corresponding full subcomplex.

\begin{lemma} \label{join}
Let $I_1,\ldots ,I_k \subset \F_p^n \setminus \{0\}$ be such that
\[
\dim \spn (I_1) + \cdots +\dim \spn (I_k) = \dim \spn (I_1 \cup \cdots \cup I_k).
\]
Then
\[
\mathcal{X}_{I_1\sqcup \cdots \sqcup I_k} \cong \mathcal{X}_{I_1}\ast \cdots \ast \mathcal{X}_{I_k}.
\]
\end{lemma}
\begin{proof}
Let $k=2$ and denote $I=I_1$ and $J=I_2$.
First note that the given condition is equivalent to $\spn (I) \cap \spn (J) = \{0\}$. We may also conclude that $I$ and $J$ are disjoint.

From the definitions we have that $\mathcal{K}_{I \sqcup J}\subset \mathcal{K}_I \ast \mathcal{K}_J$ for any simplicial complex $\mathcal{K}$ and disjoint $I$ and $J$.
On the other hand, let $\sigma$ be a simplex in $\mathcal{K}_I \ast \mathcal{K}_J$. Then $\sigma \cap I\in \mathcal{K}_I$ and $\sigma \cap J \in \mathcal{K}_J$ are simplices in $\mathcal{K}$ as well, and because of the given condition, they are independent. Therefore, $\sigma = (\sigma \cap I )\cup (\sigma \cap J)$ is a simplex in $\mathcal{K}_{I \sqcup J}$.

For $k>2$, the lemma follows by the induction.
\end{proof}

Using Baskakov's description of the product in the cohomology \cite{Baskakov02}, we are able to calculate the cup length of the cohomology of the moment-angle complex of $\mathcal{X}$ and estimate its Lusternick-Schnirelman category. Recall that the Lusternik-Schnirelmann category (or simply category) $\cat(X)$ of  a topological space $X$ is the smallest integer $k$ such that $X$ admits
a covering by $k + 1$ open sets contractible in $X$. The Lusternik-Schnirelmann category of moment angle complexes was studied by Beben, and Grbi\'{c} in \cite{BebenGrbic3}.

\begin{prop}
\label{cupX2}
For every $n\in\N$ and a prime number $p>2$ the cup length
\[
\cp \widetilde{H}^*(\mathcal{Z}_{X(\F_p^n)};\Z) =n.
\]

\end{prop}

\begin{proof}
Let $\alpha_1,\ldots,\alpha_k \in \widetilde{H}^*(\mathcal{Z}_\mathcal{X}) $ be such that their product is nonzero. Then, via Hochster's formula there are $I_1,\ldots , I_k\subset \F_p^n\setminus \{0\}$ and $r_1,\ldots , r_k \in \N_0$, such that $\alpha_i \in \widetilde{H}^{r_i}(\mathcal{X}_{I_i})$.

Using Baskakov's description of the cup product, we conclude that $I_i$ are mutually disjoint and that
\[
\alpha_1 \smile \cdots\smile \alpha_k\in  \widetilde{H}^{r_1+\cdots+r_k+k-1} (\mathcal{X}_{I_1\sqcup \dots \sqcup I_k})
\]
is nonzero. From dimensional reasons we have  $r_1+\cdots+r_k+k-1\leq n-1$ and  since $r_i\geq 0$, for all $1\leq i\leq k$, we conclude that $k\leq n$.

For the other inequality, it is sufficient to find $n$ cohomology classes in $\widetilde{H}^*(\mathcal{Z}_\mathcal{X})$ having the nontrivial product. Define $I_j=\{e_{j}, 2 e_{j}\}$, for $1\leq j\leq n$. Obviously, $\mathcal{X}_{I_j}\cong S^0$ for every $j$, so take $\alpha_j\in \widetilde{H}^0(\mathcal{X}_{I_j})$ to be the generator. From Lemma~\ref{join} we have, $$\mathcal{X}_{I_1\sqcup \dots \sqcup I_{n}}\cong \mathcal{X}_{I_1}\ast\dots \ast \mathcal{X}_{I_{n}}\cong S^{n-1}$$ and by Baskakov's formula the product $\alpha_1 \smile \ldots \smile \alpha_{n}$ is the generator of the top cohomology of $S^{n-1}$ and therefore is nonzero.
\end{proof}

\begin{thm}
For every $n\in\N$ and a prime number $p>2$ the Lusternick-Schnirelmann category
\[
\cat (\mathcal{Z}_{X(\F_p^n)}) =n.
\]
\end{thm}

\begin{proof} From a classical result we have the inequality  $\cp \widetilde{H}^*(\mathcal{Z}_{\mathcal{X}};\Z) \leq \cat (\mathcal{Z}_\mathcal{X})$.
The opposite inequality $\cat (\mathcal{Z}_\mathcal{X})\leq n$ follows from  \cite[Theorem~1.9]{BebenGrbic3}.
\end{proof}

The same argument based on the join of spheres implies.

\begin{prop}
\label{prodt}
Let $I_1,\ldots ,I_k\subset \F_p^n\setminus \{0\}$ are such that  $\mathcal{X}_{I_j}$ are simplicial spheres, for all $1\leq j\leq k$, and that
\[
\dim \spn (I_1) + \cdots +\dim \spn (I_k) = \dim \spn (I_1 \cup \cdots \cup I_k).
\]
Then the classes in $\widetilde{H}^*(\mathcal{Z}_\mathcal{X})$ which corresponds via Hochster's formula to the generators of the top cohomology of $\mathcal{X}_{I_j}$ have a nontrivial product.
\end{prop}

The proof of Proposition \ref{cupX2} is also valid for $\mathcal{Z}_{\mathcal{K}}$, but with minor technical modifications.

\begin{prop}For every $n\in\N$ and a prime number $p$ the cup length
\[
\cp \widetilde{H}^*(\mathcal{Z}_{K(\F_p^n)};\Z)=\left[\frac{n}{2}\right].
\]
\end{prop}

\begin{proof}
The complex $\mathcal{K}$ is 2-neighborly and for $|I|=2$, $\mathcal{K}_I$ is contractible, and therefore it does not produce a nontrivial cohomology class in $\widetilde{H}^*(\mathcal{Z}_\mathcal{K})$.  Moreover, if for some subset of the vertices $I$ there is a nonzero class $\alpha \in \widetilde{H}^{\ast} (\mathcal{K}_I)$ then it is in degree at least 1.  Let $\alpha_1,\ldots,\alpha_k \in \widetilde{H}^*(\mathcal{Z}_\mathcal{K}) $ be such that their product is nonzero, and let  $I_1,\ldots , I_k\subset \F_p^n\setminus \{0\}$ and $r_1,\ldots , r_k \in \N$ be  such that $\alpha_i \in \widetilde{H}^{r_i}(\mathcal{K}_{I_i})$.

Again, Baskakov's description of the cup product gives that $I_i$ are mutually disjoint and that
\[
\alpha_1 \smile \cdots \smile \alpha_k\in  \widetilde{H}^{r_1+\cdots+r_k+k-1} (\mathcal{K}_{I_1\sqcup \dots \sqcup I_k})
\]
is nonzero. From dimensional reasons we have  $r_1+\cdots+r_k+k-1\leq n-1$ and  since $r_i\geq 1$, for all $1\leq i\leq k$, we conclude that $2k\leq n$.

Now we give $\left[\frac{n}{2}\right]$ cohomology classes in $\widetilde{H}^*(\mathcal{Z}_\mathcal{K})$ having nontrivial product. Let us consider the following sets $I_j=\{l(e_{2 j-1}), l(e_{2j}), l(e_{2 j-1}+ e_{2j})\}$, where $1\leq j\leq \left[\frac{n}{2}\right]$. Obviously, $\mathcal{K}_{I_j}\cong S^1$ for every $j$, so take $\alpha_j\in H^1(\mathcal{K}_{I_j})$ to be the generator. From Corollary~\ref{join} we have, $$\mathcal{K}_{I_1\sqcup \dots \sqcup I_{\left[\frac{n}{2}\right]}}\cong \mathcal{K}_{I_1}\ast\dots \ast \mathcal{K}_{I_{\left[\frac{n}{2}\right]}}\cong S^{2\left[\frac{n}{2}\right]-1}$$ and by Baskakov's formula the product $\alpha_1 \smile \ldots \smile \alpha_{\left[\frac{n}{2}\right]}$ is the generator of the top cohomology of $S^{2\left[\frac{n}{2}\right]-1}$ and therefore is nonzero.
\end{proof}

\begin{thm}
For every $n\in\N$ and a prime number $p$ the Lusternik-Schnirelmann category
\[
\cat (\mathcal{Z}_{K(\F_p^n)})=\left[\frac{n}{2}\right].
\]
\end{thm}

\begin{proof} Because $\cat (\mathcal{Z}_\mathcal{K})\geq \mathrm{cup} \widetilde{H}^*(\mathcal{Z}_\mathcal{K};\Z)=\left[\frac{n}{2}\right]$ it is sufficient to prove that $\cat (\mathcal{Z}_\mathcal{K})\leq\left[\frac{n}{2}\right]$. Since $\mathcal{K}$ is $2$-neighbourly  from \cite[Theorem~1.9]{BebenGrbic3} we have
\[
\cat (\mathcal{Z}_\mathcal{K})\leq \frac{\dim \mathcal {K} +1}{2}=\frac{n}{2}.
\]
\end{proof}

\section{Buchstaber invariant}

Let $\mathcal{K}$ be a simplicial complex on $m$ vertices.
The torus $T^m$ acts naturally coordinatewisely on $\mathcal{Z}_\mathcal{K}$ while $_{\mathbb{R}} \mathcal{Z}_\mathcal{K}$ has the natural coordinatewise action of the real torus $\F_2^m$. The diagonal actions of $S^1$ and $S^0$ on  $\mathcal{Z}_\mathcal{K}$ and $_{\mathbb{R}} \mathcal{Z}_\mathcal{K}$ are free. It is of particular interest to find the maximal dimension of subgroups of $T^m$ and $\F_2^m$ that act freely on $\mathcal{Z}_\mathcal{K}$ and $_{\mathbb{R}} \mathcal{Z}_\mathcal{K}$, respectively.

\begin{defn}\label{bi} A \textit{complex Buchstaber invariant} $s(\mathcal{K})$ of $\mathcal{K}$ is a maximal dimension of a toric subgroup of $T^m$ acting freely on $\mathcal{Z}_\mathcal{K}$.

A \textit{real Buchstaber invariant} $s_{\mathbb{R}} (\mathcal{K})$ of $K$ is a maximal rank of a subgroup of $\F_2^m$ acting freely on $_{\mathbb{R}} \mathcal{Z}_\mathcal{K}$.
\end{defn}

It is an open problem in toric topology to find a combinatorial description of $s(\mathcal{K})$ and $s_{\mathbb{R}} (\mathcal{K})$. Buchstaber's invariants are closely related with universal simplicial complexes $K (\mathbf{\mathbb{Z}}^n)$ and $K (\mathbf{\mathbb{F}}_2^n)$. Following~\cite[Section 2]{Ayze1}, we recall a few of the properties needed for our work.

\begin{prop}\label{bin}
Let $\mathcal{K}$ be a simplicial complex on $m$ vertices.
A \textit{complex Buchstaber invariant}  of $\mathcal{K}$ is the integer $s(\mathcal{K})=m-r$, where $r$ is the least integer such that there is a nondegenerate simplicial map
\[
f\colon \mathcal{K} \rightarrow K
	(\mathbf{\mathbb{Z}}^r) .
	\] 	
A \textit{real Buchstaber invariant}  of $\mathcal{K}$ is the integer $s_{\mathbb{R}} (\mathcal{K})=m-r$, where $r$ is the least integer such that there is a nondegenerate simplicial map
	$$ f\colon \mathcal{K} \rightarrow K (\F_2^r) .
	$$\qed
\end{prop}
Recall from the section 2 that there are nondegenerate maps $\phi \colon \xkn \lra \kkn$ and $\xi \colon \kkn \lra \xkn$. Thus $X(\Z^r)$ can be used instead of $K(\Z^r)$ in Proposition~\ref{bin} to compute a complex Buchstaber invariant. Note that for the real case  $X(\F_2^r)=K(\F_2^r)$.

Finding explicit values of Buchstaber invariants for a given simplicial complex $\mathcal{K}$ is an important open problem in toric topology and in general, only some estimates of these invariants are known. A sequence $\{L^i\}_{i\in\N}$ of simplicial complexes is called an \textit{increasing sequence} if for every $i$ and $j$ such that $i<j$ there is a nondegenerate simplicial map $f\colon L^i\rightarrow L^j$. Observe that Proposition~\ref{bin} can be used to define the Buchstaber invariant $s_{{\{L^n\}}} (\mathcal{K})$
for any increasing sequence of simplicial complexes  $\{L^n\}$ instead of $\{K(\F_2^n)\}$ or $\{K({\mathbb{Z}}^n)\}$. For the sequence of standard simplices $\{\Delta^i\}$, the Buchstaber invariants are $m-\gamma(\mathcal{K})$, where $\gamma (\mathcal{K})$ is the classical chromatic number of a simplicial complex. In particular, for a prime number $p$, the family $\{K(\F_p^n)\}$ can be considered, which allows the following definition.
\begin{defn}
Let $\mathcal{K}$ be a simplicial complex on $m$ vertices and $p$ a prime number. A \emph{mod $p$ Buchstaber invariant} $s_{p}(\mathcal{K})$ is given by
\[
s_{p}(\mathcal{K})=m-r
\]
where $r$ is the least integer such that there is a nondegenerate simplicial map $f\colon\mathcal{K} \rightarrow K(\F_p^r)$.
\end{defn}

\begin{rem} Note that $s_{\mathbb{R}} (\mathcal{K})$ is mod $2$ Buchstaber invariant, that is, $s_{2} (\mathcal{K})$ in this notation.
\end{rem}

\begin{rem}
\label{ref2}
A topological description of $s_p (\mathcal{K})$ can be given as a maximal rank of subgroup of $\Z_p^m$  acting freely on  the polyhedral product $\mathcal{Z}_\mathcal{K} (\mathrm{Cone}( \Z_p), \Z_p)\subset D^{2m}$, where $\Z_p$ is considered as the set of $p$-th roots of unity. Unlike in the cases of $\F_2^n$ and $\mathbb{Z}$, $\mathcal{Z}_\mathcal{K} (\mathrm{Cone}( \Z_p), \Z_p)$ for $p>2$ fails to be a manifold when $K$ is a triangulation of a sphere.
\end{rem}

Ayzenberg in \cite {Ayze} and Erokhovets in \cite{MR3482595} studied nondegenerate maps between various increasing sequences of simplicial complexes to establish estimates of Buchstaber invariants in terms of combinatorial invariants of simplicial complexes. They obtained two interesting results.

\begin{prop}\cite[Proposition 2]{Ayze1} \label{ay1} Let $\{L^i\}$ and $\{M^i\}$ be two increasing sequences of simplicial complexes such that for each $i$, there is a nondegenerate map $f_i\colon L^i\rightarrow M^i$. Then $$s_{\{L^i\}} (\mathcal{K})\leq s_{\{M^i\}} (\mathcal{K}).$$\qed
\end{prop}

\begin{prop}\cite[Proposition 4]{Ayze1} \label{ay2} Let $\{L^i\}$ be an increasing sequence of simplicial complexes and let $\mathcal{K}_1$ and $\mathcal{K}_2$ be simplicial complexes on vertex sets $[m_1]$ and $[m_2]$, respectively. Suppose that there exists a nondegenerate map $f\colon \mathcal{K}_1 \rightarrow \mathcal{K}_2$. The following holds $$s_{\{L^i\}} (\mathcal{K}_1)-s_{\{L^i\}} (\mathcal{K}_2)\geq m_1-m_2.$$\qed
\end{prop}

Recall that the \textit{chromatic number} of a simplicial complex $\mathcal{K}$ is the minimal number $\gamma (\mathcal{K})$ such that there exists a nondegenerate map $f \colon \mathcal{K} \rightarrow \Delta^{\gamma (\mathcal{K})-1}$.

\begin{cor}\label{cor5.8}
Let $\mathcal{K}$ be a simplicial complex on $m$ vertices. The following inequalities hold
\begin{equation}
\label{in:pg}
m-\gamma (\mathcal{K}) \leq s (\mathcal{K})\leq s_{p} (\mathcal{K})\leq m-\dim{\mathcal{K}}-1.
\end{equation}
 \end{cor}
\begin{proof}
There is an inclusion of $\Delta^{n-1}$ into $K({\mathbb{Z}}^n)$ and a modulo $p$ reduction map from
$K(\mathbb{Z}^n)$ to $K(\F_p^n)$ which is nondegenerate for any prime number $p$. The statement now follows by Proposition~\ref{ay1}.
\end{proof}

\begin{prop}
For every simplicial complex $\mathcal{K}$, there is an infinite set of prime numbers denoted by $P (\mathcal{K})$ such that for all $p$, $q\in P (\mathcal{K})$ $$s_{p} (\mathcal{K})=s_{q} (\mathcal{K}).$$
\end{prop}

\begin{proof}
Using the Pigeonhole principle, relation~\eqref{in:pg} implies the claim since the set of possible values of $s_{p} (\mathcal{K})$ is a finite set, which contradicts the fact that the set of prime numbers is a countable infinite set.
\end{proof}
As an immediate corollary of Propositions \ref{ay1} and \ref{ay2}, in the same fashion as it was done for $p=2$ in \cite{Ayze1}, the following result holds.

\begin{cor} Let $\mathcal{K}$ be a simplicial complex on vertex set $[m]$ and $p$ a prime number. Then the following inequality holds
	\begin{equation}
s (\mathcal{K})\leq s_{p} (\mathcal{K})\leq m-\lceil \log_p \left((p-1)\gamma (\mathcal{K})+1 \right)\rceil.
	\end{equation}
\end{cor}

\begin{proof} The second inequality only needs to be proved. Let $r$ be a minimal integer such that there is a nondegenerate map $f\colon \mathcal{K}\rightarrow K(\F_p^r)$. There is a nondegenerate inclusion $e \colon K(\F_p^r)\rightarrow \Delta^{\frac{p^r-1}{p-1}-1}$ as $ K(\F_p^r)$ has $\frac{p^r-1}{p-1}$ vertices. Thus, the composition $ e\circ  f$ is a nondegenerate map from  $\mathcal{K}$ to $\Delta^{\frac{p^r-1}{p-1}-1}$. However, $\gamma (\mathcal{K})-1$ is the minimal dimension of a simplex into which $\mathcal{K}$ can be mapped nonregenerative so
\[
\gamma (\mathcal{K})\leq {\frac{p^r-1}{p-1}}.
\]
The desired inequality easily follows from the inequality above.
\end{proof}

Using the same argument with appropriate modifications as in Ayzenberg's proof of \cite[Theorem~1]{Ayze1}, we deduce the following result.

\begin{prop}
\label{prop}
For any simple graph $\Gamma$,
$$
s_{p} (\Gamma)=m-\left \lceil \log_p ((p-1) \gamma (\Gamma)+1)\right \rceil.
$$\qed
\end{prop}

\begin{cor}\label{graf}
Let $\Gamma$ be a simple graph on $m$ vertices and $p$ a prime number. Let $k$ be a natural number such that
\[
p^{k-2}+p^{k-3}+\ldots +1<\gamma (\Gamma) \leq p^{k-1}+p^{k-2}+\ldots +1.
\]
Then
\[
s_{p}(\Gamma)=m-k.
\]

In particular, if $\Gamma$ is not a discrete graph and $\gamma (\Gamma )\leq p+1$ then $s_{p}(\Gamma)=m-2. $
\end{cor}
\begin{proof}
For assumed $k$, we have
\[
p^{k-1} < (p-1)\gamma (\Gamma) +1 \leq p^k.
\]
Applying $\log_p$ to this inequality implies that
\[
\left \lceil \log_p ((p-1) \gamma (\Gamma)+1)\right \rceil =k.
\]
Proposition~\ref{prop} now finishes the proof.
\end{proof}

Corollary \ref{graf} implies that a simple graph $\Gamma$ can be mapped by a nondegenerate map in $K(\mathbb{F}_p^k)$ when $k$ is such that
\[
p^{k-2}+p^{k-3}+\ldots +1<\gamma (\Gamma) \leq p^{k-1}+p^{k-2}+\ldots +1.
\]
It is easy to realize a map like that. Let $f \colon V(\Gamma) \rightarrow [ \gamma(\Gamma)]$ be a proper coloring of the vertices of $\Gamma$. Since the number of vertices of $K(\mathbb{F}_p^k)$ is $p^{k-1}+\ldots +1$, there is an injection $i \colon  [ \gamma(\Gamma)] \rightarrow K(\mathbb{F}_p^k)$. The composition $i\circ f\colon V(\Gamma )\rightarrow K(\mathbb{F}_p^k)$ has an extension to a simplicial map $\overline{i\circ f}\colon \Gamma \rightarrow K(\mathbb{F}_p^k)$ because the $1$-skeleton of $K(\mathbb{F}_p^k)$ is a complete graph.

\begin{prop} \label{my1} Let $p$, $q$ be prime numbers and let $m$ and $n$ be positive integers. Let $f$ be a nondegenerate map $f\colon K(\F_p^m)\rightarrow K(\F_q^n)$. Then $f$ is an injection.
\end{prop}

\begin{proof} By dimensional reasons, if $f$ is a nondegenerate map $m$ must not be greater than $n$. Observe that the $1$-skeletons $K^{(1)}(\F_p^m)$ and  $K^{(1)}(\F_q^n)$ are complete graphs, so nondegeneracy of $f$ implies that $f$ is an injection on the set of vertices of $K(\F_p^m)$. It follows that $f$ is an injection since it is a simplicial map.
\end{proof}

The following is an instant consequence of Proposition \ref{my1}.

\begin{cor}\label{my2}
If there is a nondegenerate map $f\colon K(\F_p^m)\rightarrow K(\F_q^n)$,  then
    \begin{eqnarray*}
    m &\leq & n,\\
f_i \left(  K(\F_p^m)\right)& \leq & f_i \left(  K(\F_q^n)\right)  \text{\,\, for all\,\, } i.
\end{eqnarray*}
\end{cor}

For any two prime numbers $p$ and $q$, define a function $\omega_{p, q}(n)$ that assigns to each positive integer $n$ the minimal integer $r$ such that there is a nondegenerate map from $K(\F_p^n)$ to $K(\F_q^r)$. Similarly, for a prime number $p$ define a function $\theta_p(n)$ such that $\theta_p (n)$ is the minimal integer $r$ such that there is a nondegenerate map from $K(\F_p^n)$ to $K(\mathbb{Z}^r)$.

\begin{prop} The functions $\omega_{p, q}(n)$ and $\theta_{p}(n)$ are increasing.
\end{prop}

\begin{proof} The statement follows from the inclusion of $K({\F_p}^n)$ into $K({\F_p}^{n+1})$ and Proposition \ref{ay1}.
\end{proof}

\begin{thm}\label{bounds1}
For every $n\in\N$ and every prime numbers $p$ and $q$ we have
$$\left\lceil \log_q \left(\frac{(q-1)(p^n-1)}{p-1}+1\right) \right\rceil \leq \omega_{p, q}(n)\leq \frac{p^n-1}{p-1}.$$
\end{thm}
\begin{proof}
Corollary \ref{my2} implies the following relation \begin{equation}\label{ob1}
\frac{p^n-1}{p-1}\leq \frac{q^{\omega_{p,q} (n)}-1}{q-1}.
\end{equation}
On the other hand, $K(\F_p^n)$ includes into $K({\F_q}^{\frac{p^n-1}{p-1}})$ via the composition of the inclusions
\begin{equation}\label{ob2}
K(\F_p^n)\xhookrightarrow{} \Delta^{\frac{p^n-1}{p-1}-1}\xhookrightarrow{} K({\F_q}^{\frac{p^n-1}{p-1}}).
\end{equation}
We deduce the desired inequalities from the \eqref{ob1} and \eqref{ob2}.
\end{proof}

By Theorem~\ref{bounds1}, the function $\omega_{p, q}(n)$ is bounded above with $ \frac{p^n-1}{p-1}$ implying that it is well defined.

For any positive integer $n$, there is a composition of nondegenerate maps
\[
K(\F_p^n)\rightarrow K(\Z^{\theta_p (n)})\rightarrow K(\F_q^{\theta_{p}(n)})
\]implying the following statement.

\begin{prop}
\label{estimate} For every positive integer $n$ and any two prime numbers $p$ and $q$, the following inequality holds $$\theta_p (n)\geq \omega_{p, q} (n).$$\qed
\end{prop}
 Using  Proposition~\ref{estimate}, the lower bound of $\omega_{p, q} (n)$ given in Theorem~\ref{bounds1}  is also a lower bound of $\theta_p (n)$. The largest lower bound is reached for $q=2$, so
\[
\theta_p (n)\geq \left\lceil \log_2 \left(\frac{p^n-1}{p-1}+1\right) \right\rceil.
\]
In a similar way, using~\eqref{ob2}, we have
\[
\theta_p (n)\leq \frac{p^n-1}{p-1}.
\]
\subsection{Some calculations of $s_p$}

The skeletons of simplices are a pivotal class of simplicial complexes for determining the Buchstaber number of an arbitrary simplicial complex. Fukukawa and Masuda studied the real Buchstaber invariant of the skeleton of simplex in \cite{FukMas}. The subsequent results extend theirs for the mod $p$ Buchstaber invariants, but we prove them using another method.

Let $\Delta^m$ be the $m$-simplex on the set of vertices $[m+1]$ and let $\Delta_{(k)}^m$ be its $k$-skeleton.

\begin{prop}A nondegenerate simplicial map $f \colon \Delta_{(k)}^{m} \longrightarrow X(\F_p^n)$ is an embedding if $k\geq 1$.
\end{prop}
\begin{proof}
Every two vertices in $\Delta_{(k)}^{m}$
are connected by an edge, so the images by a nondegenerate map must be distinct. Therefore $f$ is an embedding.
\end{proof}

\begin{prop}\label{emb}
For every prime number $p$ and every $0\leq k\le m$ we have
\[
s_p(\Delta_{(k+1)}^{m+1})\leq s_p(\Delta_{(k)}^{m})\leq s_p(\Delta_{(k)}^{m+1}) \leq s_p(\Delta_{(k)}^{m}) +1.
\]
\end{prop}

\begin{proof}
If $k=0$ then the chromatic number $\gamma (\Delta_{(0)}^{m})=1$ so by Corollary \ref{cor5.8} we have $s_p (\Delta_{(0)}^{m})=m$. From Proposition \ref{prop} $s_p (\Delta_{(1)}^{m})=m$ so we verify that all inequalities hold in this case. Let us consider the case $k\geq 1$/

Let $f \colon \Delta_{(k+1)}^{m+1} \rightarrow \mathcal{X}^n$ be an embedding. Without loss of generality, we may assume that $f(m+2)=e_n$.
Let us consider a map $\bar{f} \colon \Delta_{(k)}^m \rightarrow \mathcal{X}^{n-1}$ such that $\bar{f}(i) = \pi (f(i))$ where $\pi \colon \F_p^n\rightarrow \F_p^{n-1}$ is the projection on the first $n-1$ coordinates. We will prove that $\bar{f}$ is nondegenerate.

Let us assume the contrary. Then there exist distinct $1 \leq i_1< i_2< \dots< i_j \leq  m + 1$ such that
$j \leq k+1$ and $a_1, a_2, \dots, a_j\in \F_p$ not all being zero such that $a_1\bar{f}(i_1)+\cdots+ a_j\bar{f}(i_j) = (0, \dots, 0)\in \F_p^{n-1}$. Because of nondegeneracy of $f$ we know that $b_1f(i_1)+b_2f(i_2)+\cdots+b_jf(i_j) = (0, \dots, 0)\in \F_p^{n}$, if and only if $b_i=0$ for all $i$, so
$\pi (a_1f(i_1)+a_2f(i_2)+\cdots+a_jf(i_j))= a_1\bar{f}(i_1)+\cdots+ a_j\bar{f}(i_j) = (0, \dots, 0)$ implies that
$a_1f(i_1)+a_2f(i_2)+\cdots+a_jf(i_j)=r e_n$ for some $1\leq r \leq p-1$. Since $j + 1 \leq k + 2$ the
vectors $f(i_1), \dots, f(i_j), f(m + 2)=e_n$ do not span a simplex in $\mathcal{X}^n$ contradicting a
nondegeneracy of $f$.
Thus, we constructed a nondegenerate map from $\Delta_{(k)}^m$ to $\mathcal{X}^{n-1}$. The minimal $r$ such that there exists nondegenerate map $\Delta_{(k)}^m\to \mathcal{X}^r$ is less then $n$ so $s_p(\Delta_{(k)}^m)\ge m+1-(n-1)=s_p(\Delta_{(k+1)}^{m+1})$.

Let $f\colon \Delta_{(k)}^{m}\longrightarrow \mathcal{X}^n$ be a nondegenerate simplicial map. We consider $\mathcal{X}^n$ as a subcomplex of $\mathcal{X}^{n+1}$ by adding $0$ as the last coordinate to all vertices. We also consider $\Delta_{(k)}^m$ as a subcomplex of $\Delta_{(k)}^{m+1}$ by adding vertex $m+2$. Now we define an extension $f'$ of $f$ by defining $f'(m+2)= e_{n+1}$. The map $f'\colon \Delta_{(k)}^{m+1} \longrightarrow \mathcal{X}^{n+1}$ is a nondegenerate simplicial map, and the existence of such map gives the left inequality.

The last inequality follows from Proposition~\ref{ay1} since $\Delta_{(k)}^m$ is a subcomplex of $\Delta_{(k)}^{m+1}$ and if we have a nondegenerate map $\Delta_{(k)}^{m+1} \longrightarrow \mathcal{X}^n$ then we have such a map from its subcomplex.
\end{proof}

\begin{center}\textmd{\textbf{Acknowledgements} }
\end{center}

\medskip
Research on this paper was partially supported by the bilateral project ``Discrete Morse theory and its Applications" funded by the Ministry for Education and Science of the Republic of Serbia and the Ministry of Education, Science and Sport of the Republic of Slovenia. The third author was supported by the Slovenian Research Agency, Grant No. P1-0292. All authors are most grateful to Jelena Grbi\'{c} for various comments, discussions, and contributions to the article.

\bibliographystyle{amsplain}
%\bibliography{mybibliography}
\providecommand{\bysame}{\leavevmode\hbox to3em{\hrulefill}\thinspace}
\providecommand{\MR}{\relax\ifhmode\unskip\space\fi MR }
\providecommand{\MRhref}[2]{%
  \href{http://www.ams.org/mathscinet-getitem?mr=#1}{#2}
}
\providecommand{\href}[2]{#2}

\end{document}